\documentclass{aims-ppn}
\usepackage{amsmath,amsfonts,amssymb,bbm}
\usepackage[colorlinks=true]{hyperref}
\hypersetup{urlcolor=blue, citecolor=red}

\newtheorem{theorem}{Theorem}[section]
\newtheorem{corollary}{Corollary}

\newtheorem{lemma}[theorem]{Lemma}
\newtheorem{proposition}{Proposition}

\theoremstyle{definition}

\newtheorem{remark}{Remark}

\newcommand{\be}[1]{\begin{equation}\label{#1}}
\newcommand{\ee}{\end{equation}}
\renewcommand{\(}{\left(}
\renewcommand{\)}{\right)}

\newcommand{\bangle}[1]{\langle #1\rangle}

\newcommand{\var}{\varepsilon}
\newcommand\dt{{\frac{\mathrm d}{\mathrm dt}}}

\newcommand{\R}{{\mathbb R}}
\newcommand{\N}{{\mathbb N}}

\newcommand{\ird}[1]{\int_{\R^d}{#1}\,dx}

\newcommand{\irdxi}[1]{\int_{\R^d}{#1}\,d\xi}
\newcommand{\irdmuxi}[2]{\int_{\R^d}#1\,d\mu_{#2}}
\newcommand{\nrm}[2]{\left\|{#1}\right\|_{#2}}
\newcommand{\nrmadhoc}[2]{\left\|{#1}\right\|_{#2}}
\newcommand{\nrmV}[2]{\left\|{#1}\right\|_{\mathrm L^{#2}(e^Vdx)}}

\newcommand{\eps}{\varepsilon}
\newcommand{\aalpha}{{\alpha_\star}}
\newcommand{\bbeta}{{\beta_\star}}
\newcommand{\ggamma}{{\gamma_\star}}
\newcommand{\op}[1]{\mathsf{#1}}
\newcommand{\scalar}[2]{\left\langle{#1},{#2}\right\rangle}

\title[Diffusion with very weak confinement]{Diffusion with very weak confinement}

\author[Emeric Bouin and Jean Dolbeault and Christian Schmeiser]{}


\email[E.~Bouin]{bouin@ceremade.dauphine.fr}
\email[J.~Dolbeault]{dolbeaul@ceremade.dauphine.fr}
\email[C.~Schmeiser]{Christian.Schmeiser@univie.ac.at}

\thanks{$^*$ Corresponding author: Emeric Bouin}

\begin{document}
\date{\today}
\maketitle

\centerline{\scshape Emeric Bouin$^*$}
\medskip
{\footnotesize\centerline{CEREMADE (CNRS UMR n$^\circ$ 7534), PSL university}
\centerline{Universit\'e Paris-Dauphine, Place de Lattre de Tassigny, 75775 Paris 16, France}}
\medskip

\centerline{\scshape Jean Dolbeault}
\medskip
{\footnotesize\centerline{CEREMADE (CNRS UMR n$^\circ$ 7534), PSL university}
\centerline{Universit\'e Paris-Dauphine, Place de Lattre de Tassigny, 75775 Paris 16, France}}
\medskip

\centerline{\scshape Christian Schmeiser}
\medskip
{\footnotesize\centerline{Fakult\"at f\"ur Mathematik, Universit\"at Wien}
\centerline{Oskar-Morgenstern-Platz 1, 1090 Wien, Austria}}

\medskip

%%%%%%%%%%%%%%%%%%%%%%%%%%%%%%%%%%%%%%%%%%%%%%%%%%%%%%%%%%%%%%%%%%%%%%
\begin{abstract} This paper is devoted to Fokker-Planck and linear kinetic equations with very weak confinement corresponding to a potential with an at most logarithmic growth and no integrable stationary state. Our goal is to understand how to measure the decay rates when the diffusion wins over the confinement although the potential diverges at infinity.\end{abstract}
\thispagestyle{empty}

%%%%%%%%%%%%%%%%%%%%%%%%%%%%%%%%%%%%%%%%%%%%%%%%%%%%%%%%%%%%%%%%%%%%%%
%%%%%%%%%%%%%%%%%%%%%%%%%%%%%%%%%%%%%%%%%%%%%%%%%%%%%%%%%%%%%%%%%%%%%%
\section{Introduction}

This paper addresses the large time behavior of the solutions to the macroscopic Fokker-Planck equation and to kinetic equations with Fokker-Planck or scattering collision operators.

\medskip The first part of this paper deals with the \emph{macroscopic Fokker-Planck equation}
\be{eq:FPmacro}
\frac{\partial u}{\partial t}=\Delta_x u+\nabla_x\cdot(\nabla_xV\,u)=\nabla_x\(e^{-V}\,\nabla_x\(e^V\,u\)\)
\ee
where $x\in \R^d$, $d\ge3$, and $V$ is a potential such that $e^{-V}\not\in\mathrm L^1(\R^d)$, that is, $e^{-V}\,dx$ is an \emph{unbounded invariant measure}. We shall investigate the two following examples
\[
V_1(x)=\gamma\,\log|x|\quad\text{and}\quad V_2(x)=\gamma\,\log\bangle{x}
\]
with $\gamma<d$ and $\bangle{x}:=\sqrt{1+|x|^2}$ for any $x\in \R^d$. These two potentials share the same asymptotic behavior as $|x|\to\infty$. The potential $V_1$ is invariant under scalings, whereas $V_2$ is smooth at the origin. In both cases, the only integrable equilibrium state is $0$. Thus, if the initial datum $u_0$ is such that $u_0\in \mathrm L^1(\R^d)$, we expect that the solution to~\eqref{eq:FPmacro} converges to $0$ as $t\to+\infty$. When $\gamma>0$, the potential $V$ is \emph{very weakly confining} in the sense that, even if it eventually slows down the decay rate, it is not strong enough to produce a stationary state of finite mass: the diffusion wins over the drift. Our goal to establish the rate of convergence in suitable norms. We shall use the notation $\nrm{\cdot}p:=\nrmadhoc{\cdot}{\mathrm L^p(dx)}$ in case of Lebesgue's measure and specify the measure otherwise.
%---------------------------------------------------------------------
\begin{theorem}\phantomsection\label{thm:macrounw} Assume that $d\ge3$, $\gamma<(d-2)/2$ and $V=V_1$ or $V=V_2$. Then any solution $u$ of~\eqref{eq:FPmacro} with initial datum $u_0\in\mathrm L^1_+\cap\mathrm L^2(\R^d)$
satisfies, for all $t\ge 0$,
\be{Decay1}
\nrm{u(t,\cdot)}2^2\le\frac{\nrm{u_0}2^2}{(1+c\,t)^{\frac d2}}\quad\mbox{with}\quad c:=\frac4d\,\min\left\{1,1-\tfrac{2\,\gamma}{d-2}\right\}\,\mathcal C_{\mathrm{Nash}}^{-1}\,\frac{\nrm{u_0}2^{4/d}}{\nrm{u_0}1^{4/d}}\,.
\ee
\end{theorem}
%---------------------------------------------------------------------
Here $\mathcal C_{\mathrm{Nash}}$ denotes the optimal constant in Nash's inequality~\cite{Nash,MR1230297}
\be{Ineq:Nash}
\nrm u2^{2+\frac4d}\le\mathcal C_{\mathrm{Nash}}\,\nrm u1^\frac4d\nrm{\nabla u}2^2\quad\forall\,u\in\mathrm L^1\cap\,\mathrm H^1(\R^d)\,.
\ee
Note that the rate of decay is independent of $\gamma$ and we recover the classical estimate due to J.~Nash when $V=0$ (here $\gamma=0$). The proof of Theorem~\ref{thm:macrounw} and further considerations on optimality are collected in Section~\ref{Sec:SectionL2}.

Theorem~\ref{thm:macrounw} does not cover the interval $(d-2)/2<\gamma<d$. This range is covered by employing the natural setting of $\mathrm L^2\big(e^V\big)$ and by requiring additional moment bounds.
%---------------------------------------------------------------------
\begin{theorem}\phantomsection\label{thm:macrow} Let $d\ge3$, $\gamma<d$, $V=V_1$ or $V=V_2$, and $u_0\in\mathrm L^1_+\cap\mathrm L^2\big(e^V\big)$. If $\gamma>0$, let us assume that $\nrm{|x|^k u_0}1<\infty$ for some $k\ge\max\{2,\gamma/2\}$. Then any solution of~\eqref{eq:FPmacro} with initial datum $u_0$ satisfies
\[\label{Decay2}
\forall\,t\ge0\,,\quad\nrmV{u(t,\cdot)}2^2\le\nrmV{u_0}2^2\,(1+c\,t)^{-\frac{d-\gamma}2}\,.
\]
The constant $c$ depends on $d$, $\gamma$, $k$, $\nrmV{u_0}2$, $\nrm{u_0}1$, and $\nrm{|x|^k u_0}1$.
\end{theorem}
%---------------------------------------------------------------------
The proof of Theorem~\ref{thm:macrow} is done in Section~\ref{Sec:Theorem2}. Although this is a side result, let us notice that the case in which the potential contributes to the decay, \emph{i.e.}, when $\gamma<0$, is also covered in Theorem~\ref{thm:macrow}. The scale invariance of~\eqref{eq:FPmacro} with $V=V_1$ can be exploited to obtain intermediate asymptotics in self-similar variables. Let us define
\be{ustar1}
u_\star(t,x)=\frac{c_\star}{(1+2\,t)^\frac{d-\gamma}2}\,|x|^{-\gamma} \exp\(-\frac{|x|^2}{2\,(1+2\,t)}\)\,,
\ee
The following result on \emph{intermediate asymptotics} allows us to identify the leading order term of the solution of~\eqref{eq:FPmacro} as $t\to+\infty$. It is the strongest of our results on~\eqref{eq:FPmacro} but initial data need to have a sufficient decay as $|x|\to\infty$.
%---------------------------------------------------------------------
\begin{theorem}\phantomsection\label{thm:weighted} Let $d\ge1$, $\gamma\in(0,d)$ and $V=V_1$. If for some constant $K>1$, the function $u_0$ is such that
\[
\forall\,x\in\R^d\,,\quad0\le u_0(x)\le K\,u_\star(0,x)
\]
where $c_\star$ is chosen such that $\|u_\star\|_1=\|u_0\|_1$
then the solution $u$ of~\eqref{eq:FPmacro} with initial datum $u_0$ satisfies
\[
\forall\,t\ge0\,,\quad\nrm{u(t,\cdot)-u_\star(t,\cdot)}p\le K\,c_\star^{1-\frac1p}\|u_0\|_1^\frac1p\(\tfrac e{2\,|\gamma|}\)^{\frac\gamma2\,\big(1-\tfrac1p\big)}\,(1+2\,t)^{-\zeta_p}
\]
for any $p\in[1,+\infty)$, where $\zeta_p:=\tfrac d2\,\big(1-\tfrac1p\big)+\tfrac1{2\,p}\,\min\big\{4,\,4\,(d-\gamma),\,d-1\big\}$.\end{theorem}
%---------------------------------------------------------------------
More detailed results will be stated in Section~\ref{Sec:Decay}. Let us quote some relevant papers for~\eqref{eq:FPmacro}. In the case without potential, the decay rates of the heat equation is known for more than a century and goes back to~\cite{zbMATH05828987}. Standard techniques use the Fourier transform, Green kernel estimates and integral representations: see for instance~\cite{MR2597943}. There are many other parabolic methods which provide decay rates and will not be reviewed here like, for instance, the Maximum Principle, Harnack inequalities and the parabolic regularity theory: see for instance~\cite{MR3588125}.

In his celebrated paper~\cite{Nash}, J.~Nash was able to reduce the question of the decay rates for the heat equation to~\eqref{Ineq:Nash}: see~\cite{bouin:hal-01940110} for detailed comments on the optimality of such a method. \emph{Entropy methods} have raised a considerable interest in the recent years, but the most classical approach based on the so-called \emph{carr\'e du champ method} applies to~\eqref{eq:FPmacro} only for potentials $V$ with convexity properties and a sufficient growth at infinity: typically, if $V(x)=|x|^\alpha$, then $\alpha\ge1$ is required for obtaining a \emph{Poincar\'e inequality} and the rate of convergence to a unique stationary solution is then exponential, when measured in the appropriate norms; see~\cite{MR3155209} for a general overview. An interesting family of \emph{weakly confining} potentials is made of functions $V$ with an intermediate growth, such that $e^{-V}$ is integrable but $\lim_{|x|\to\infty}V(x)/|x|=0$: all solutions of~\eqref{eq:FPmacro} are attracted by a unique stationary solution, but the rate is expected to be algebraic rather than exponential. A typical example is $V(x)=|x|^\alpha$ with $\alpha\in(0,1)$. The underlying functional inequality is a \emph{weak Poincar\'e inequality}: see~\cite{MR1856277,Kavian}, and~\cite{MR2381160} for related Lyapunov type methods \emph{\`a la} Meyn and~Tweedie or~\cite{2018arXiv180508557B} for recent spectral considerations. We refer to~\cite{MR2949623} and~\cite{MR1736202,MR1934605,MR2640175} for further considerations on, respectively, \emph{weighted Nash inequalities} and spectral properties of the diffusion operator.

\medskip The second part of this paper is devoted to \emph{kinetic equations} involving a degenerate diffusion operator acting only on the velocity variable or scattering operators, for \emph{very weak potentials} like $V_1$ or~$V_2$. Various \emph{hypocoercivity} methods have been developed over the years in, \emph{e.g.},~\cite{Herau,MR2294477,mouhot2006quantitative,MR2562709,MR3324910}, in order to prove exponential rates in appropriate norms, in presence of a \emph{strongly confining potential}. In that case, the growth of the potential at infinity has to be fast enough not only to guarantee the existence of a stationary solution but also to provide macroscopic coercivity properties which typically amount to a Poincar\'e inequality. A popular simplification is to assume that the position variable is limited to a compact set, for example a torus. Such results are the counterpart in kinetic theory of diffusions covered by the \emph{carr\'e du champ method}, as emphasized in~\cite{MR3677826}.

Recently, hypocoercivity methods have been extended in~\cite{BDMMS} to the case without any external potential by replacing the Poincar\'e inequality by Nash type estimates. The \emph{sub-exponential} regime or the regime with \emph{weak confinement}, \emph{i.e.}, of a potential~$V$ such that a weak Poincar\'e inequality holds, has also been studied in~\cite{2018arXiv180110354C,Hu_2019}. What we will study next is the range of \emph{very weak potentials}~$V$, which have a growth at infinity which is below the range of weak Poincar\'e inequalities, but are still such that $\lim_{|x|\to\infty}V(x)=+\infty$. This regime is the counterpart at kinetic level of the results of Theorems~\ref{thm:macrounw},~\ref{thm:macrow} and~\ref{thm:weighted}. As in the case of~\eqref{eq:FPmacro} when $\gamma\ge0$, the drift is opposed to the diffusion, but it is not strong enough to prevent that the solution locally vanishes.

\medskip Let us consider the kinetic equation
\be{eq:kin}
\partial_t f+v\cdot\nabla_xf-\nabla_xV\cdot\nabla_vf=\op Lf
\ee
where $\op Lf$ is one of the two following collision operators:
\begin{enumerate}
\item[(a)] a Fokker-Planck operator
\[
\op Lf=\nabla_v\cdot\Big(M\,\nabla_v\(M^{-1}\,f\)\Big)\,,
\]
\item[(b)] a scattering collision operator
\[
\op Lf=\int_{\R^d}\sigma(\cdot,v')\,\big(f(v')\,M(\cdot)-f(\cdot)\,M(v')\big)\,dv'\,.
\]
\end{enumerate}
We consider the case of a global equilibrium of the form
\[
\forall\,(x,v)\in\R^d\times\R^d\,,\quad\mathcal M(x,v)=M(v)\,e^{-V(x)}\quad\mbox{where}\quad M(v)=(2\pi)^{-\frac d2}\,e^{-\frac12\,|v|^2}\,.
\]
We shall say that the gaussian function~$M(v)$ is the \emph{local equilibrium} and assume that the \emph{scattering rate} $\sigma(v,v')$ satisfies
\begin{align*}
&{\bf (H1)}\quad1\le\sigma(v,v')\le\overline\sigma\,,\quad\forall\,v\,,\,v'\in\R^d\,,\quad\text{for some}\quad\overline\sigma\ge1\,,\\
&{\bf (H2)}\quad\int_{\R^d}\big(\sigma(v,v')-\sigma(v',v)\big)\,M(v')\,dv'=0\quad\forall\,v\in\R^d\,.
\end{align*}
Notice that $\mathcal M\not\in\mathrm L^1(\R^d\times\R^d)$ if $V=V_1$ or $V=V_2$, so that the space $\mathrm L^2\(\mathcal M^{-1}dx\,dv\)$ is defined with respect to an \emph{unbounded measure}. As in the case of~\eqref{eq:FPmacro}, the only integrable equilibrium state is $0$. Thus, if the initial datum $f_0$ is such that \hbox{$f_0\in\mathrm L^1(dx\,dv)$}, we expect that the solution to~\eqref{eq:kin} converges to~$0$ locally as $t\to+\infty$ and look for the rate of convergence in suitable norms.

When $V=0$, the optimal rate of convergence of a solution $f$ of~\eqref{eq:kin} with initial datum $f_0$ is known. In~\cite{BDMMS}, it has been proved that there exists a constant $C>0$ such that
\[
\iint_{\R^d\times\R^d}\left|f(t,\cdot,\cdot)\right|^2\,d\mu\le C\,(1+t)^{-\frac d2}\iint_{\R^d\times\R^d}\left|f_0\right|^2\,d\mu
\quad\forall\,t\ge0\,,
\]
where $d\mu=M^{-1}\,dx\,dv$ and by factorization, the result is extended with same rate for an arbitrary $\ell>d$ to the measure $\bangle v^\ell\,dx\,dv$ if $f_0\in\mathrm L^2\big(\R^d\times\R^d,\bangle v^\ell dx\,dv\big)\cap\mathrm L^2_+\(\R^d,\bangle v^\ell dv;\,\mathrm L^1\big(\R^d,dx\big)\)$. Our main result on~\eqref{eq:kin} is a decay rate in the presence of a very weak potential. It is an extension of the results of Theorem~\ref{thm:macrow} to the framework of kinetic equations.
%---------------------------------------------------------------------
\begin{theorem}\phantomsection\label{thm:kinw} Let $d\ge3$, $V=V_2$ with $\gamma\in[0,d)$ and $k>\max\left\{2,\gamma/2\right\}$. We assume that {\rm (H1)}--{\rm (H2)} hold and consider a solution $f$ of~\eqref{eq:kin} with initial datum \hbox{$f_0\in\mathrm L^2(\mathcal M^{-1}dx\,dv)$} such that $\iint_{\R^d\times\R^d}\bangle x^k\,f_0\,dx\,dv+\iint_{\R^d\times\R^d}|v|^k\,f_0\,dx\,dv<+\infty$. Then there exists $C>0$ such that
\[
\forall\,t\ge0\,,\quad\nrmadhoc{f(t,\cdot,\cdot)}{\mathrm L^2(\mathcal M^{-1}dx\,dv)}^2\le C\(1+\,t\)^{-\frac{d-\gamma}2}\,.
\]
\end{theorem}
%---------------------------------------------------------------------
Standard methods of kinetic theory can be used to establish the existence of solutions of~\eqref{eq:kin} when $V=V_2$. We will not give details here. At formal level, similar results can be expected when $V=V_1$ but the singularity at $x=0$ raises difficulties which are definitely outside of the scope of this paper.

The expression of the constant $C$ is explicit. However, due to the method, we cannot claim optimality in the estimate of Theorem~\ref{thm:kinw}, but at least the asymptotic rate is expected to be optimal by consistency with the diffusion limit, as it is the case when $V=0$ studied in~\cite{BDMMS}. The strategy of the proof and further relevant references will be detailed in Section~\ref{Sec:kinetic}.

%%%%%%%%%%%%%%%%%%%%%%%%%%%%%%%%%%%%%%%%%%%%%%%%%%%%%%%%%%%%%%%%%%%%%%
%%%%%%%%%%%%%%%%%%%%%%%%%%%%%%%%%%%%%%%%%%%%%%%%%%%%%%%%%%%%%%%%%%%%%%
\section{Decay estimates for the macroscopic Fokker-Planck equation}\label{Sec:SectionFP}

In this section, we establish decay rates for~\eqref{eq:FPmacro} and discuss the optimal range of the parameters.

%%%%%%%%%%%%%%%%%%%%%%%%%%%%%%%%%%%%%%%%%%%%%%%%%%%%%%%%%%%%%%%%%%%%%%
\subsection{Decay in \texorpdfstring{$\mathrm L^2(\R^d)$}{L2}}\label{Sec:SectionL2}
We prove Theorem~\ref{thm:macrounw}. By testing~\eqref{eq:FPmacro} with $u$, we obtain
\[\label{eqn}
\frac d{dt}\ird{u^2}=-\,2\ird{|\nabla u|^2}+\ird{\Delta V\,|u|^2}\,,
\]
with
\[
\Delta V_1(x)=\gamma\,\frac{d-2}{|x|^2}\quad\text{and}\quad\Delta V_2(x)=\gamma\,\frac{d-2}{1+|x|^2}+\frac{2\,\gamma}{\(1+|x|^2\)^2}\,.
\]
For $\gamma\le0$ we deduce
\[
\frac d{dt}\nrm{u}2^2\le-\,2\nrm{\nabla u}2^2\le-\,\frac{2}{\mathcal{C}_{\mathrm{Nash}}} \nrm{u_0}1^{-4/d} \nrm{u}2^{2+4/d}\,,
\]
from Nash's inequality~\eqref{Ineq:Nash}. Integration completes the proof of~\eqref{Decay1}. For the case $0<\gamma<(d-2)/2$ we use the following Hardy-Nash inequalities.
%---------------------------------------------------------------------
\begin{lemma}\phantomsection\label{lem:HN} Let $d\ge3$ and $\delta<(d-2)^2/4$. Then
\be{Ineq:Hardy-Nash}
\nrm u2^{2+\frac4d}\le\mathcal C_\delta \( \nrm{\nabla u}2^2 -\delta\ird{\frac{u^2}{|x|^2}}\)\,\nrm u1^\frac4d\quad\forall\,u\in\mathrm L^1\cap\,\mathrm H^1(\R^d)\,,
\ee
with
\[
\mathcal C_\delta=\mathcal{C}_{\mathrm{Nash}}\left(1-\tfrac{4\,\delta}{(d-2)^2}\right)^{-1}\,.
\]
Let additionally $\eta<(d^2-4)/4$. Then, for any $u\in\mathrm L^1\cap\,\mathrm H^1(\R^d)$,
\be{Ineq:Hardy-Nash2}
\nrm u2^{2+\frac4d}\le\mathcal C_{\delta,\eta} \( \nrm{\nabla u}2^2 -\delta\ird{\frac{u^2}{\bangle x^2}}-\eta\ird{\frac{u^2}{\bangle x^4}}\)\,\nrm u1^\frac4d
\ee
with
\[
\mathcal C_{\delta,\eta}=\mathcal{C}_{\mathrm{Nash}}\left(\min\left\{1-\tfrac{4\,\delta}{(d-2)^2},1-\tfrac{4\,\eta}{d^2-4}\right\}\right)^{-1}\,.
\]
\end{lemma}
%---------------------------------------------------------------------
The proof of Lemma~\ref{lem:HN} is given in Appendix~\ref{Appendix:C}. We use Lemma~\ref{lem:HN} with $\delta=\gamma\,(d-2)/2$ and with $\eta=\gamma$ (for $V=V_2$), and proceed as for $\gamma\le 0$ to complete the proof of Theorem~\ref{thm:macrounw}.\qed

%---------------------------------------------------------------------
\begin{remark} The condition $\delta<(d-2)^2/4$ in Lemma~\ref{lem:HN} is optimal for~\eqref{Ineq:Hardy-Nash} and~\eqref{Ineq:Hardy-Nash2}. The restriction on $\gamma$ in Theorem~\ref{thm:macrounw} is also optimal. Let $d\ge3$, $\gamma>(d-2)/2$ and $V=V_1$ or $V=V_2$. Then there exists $u\in\mathrm L^1\cap\,\mathrm H^1(\R^d)$ such that $\nrm u 2=1$ and
\[
-\,2\ird{|\nabla u|^2}+\ird{\Delta V\,|u|^2}>0\,.
\]
In the case $V=V_1$, it is indeed enough to observe that $(d-2)^2/4$ is the optimal constant in Hardy's
inequality (see Appendix~\ref{Appendix:C}). The case $V=V_2$ follows from the case $V=V_1$ by an appropriate scaling.
\end{remark}
%---------------------------------------------------------------------

%%%%%%%%%%%%%%%%%%%%%%%%%%%%%%%%%%%%%%%%%%%%%%%%%%%%%%%%%%%%%%%%%%%%%%
\subsection{Decay in \texorpdfstring{$\mathrm L^2(e^V\,dx)$}{L2(eVdx)}}\label{Sec:Theorem2}
We prove Theorem~\ref{thm:macrow}. By testing~\eqref{eq:FPmacro} with $u\,e^V$, we obtain
\be{weighted-dissipation}
\frac12\,\frac d{dt}\ird{u^2\,e^V}=-\ird{e^{-V}\,\left|\nabla\(u\,e^V\)\right|^2}\,.
\ee
In the case $V=V_1$, we have $e^V=|x|^\gamma$ and~\eqref{weighted-dissipation} takes the form
\[
\frac12 \frac d{dt}\ird{|x|^\gamma u^2}=-\ird{|x|^{-\gamma}\,\left|\nabla\(|x|^\gamma u\)\right|^2}\,.
\]
If $\gamma\le0$ and $a=\frac{d-\gamma}{d+2-\gamma}$, the inequality
\be{CKN2.0}
\ird{|x|^\gamma\,u^2}\le\mathcal C\(\ird{|x|^{-\gamma}\,\left|\nabla\(|x|^\gamma u\)\right|^2}\)^a\(\ird{|u|}\)^{2(1-a)}
\ee
follows from the Caffarelli-Kohn-Nirenberg inequalities (see Appendix~\ref{Appendix:A}, Ineq.~\eqref{CKN2moment} applied with $k=0$ to $v=|x|^\gamma u$). The conservation of the $\mathrm L^1$ norm of $u$ gives
\[
\frac d{dt}\ird{|u|^2\,|x|^\gamma}\le-\,2\,\mathcal C^{-\(1+\frac2{d-\gamma}\)}\,\nrm{u_0}1^{-\frac4{d-\gamma}}\(\ird{|u|^2\,|x|^\gamma}\)^{^{1+\frac2{d-\gamma}}}\,.
\]
The conclusion of Theorem~\ref{thm:macrow} follows by integration. An analogous argument based on the inhomogeneous Caffarelli-Kohn-Nirenberg inequality
\begin{multline*}
\ird{|u|^2\,\bangle x^\gamma}\le\mathcal K\(\ird{\bangle x^{-\gamma}\,\left|\nabla\(\bangle x^\gamma u\)\right|^2}\)^a\(\ird{|u|}\)^{2(1-a)}\\
\mbox{with}\quad a=\frac{d-\gamma}{d+2-\gamma}
\end{multline*}
applies to the case $\gamma\le0$, $V=V_2$ (see Appendix~\ref{Appendix:B}, Ineq.~\eqref{CKN2momentBis} applied with $k=0$ and $v=\bangle x^\gamma u$).

Without additional assumptions, it is not possible to expect a similar result for $\gamma>0$. Let us explain why. In the case $V=V_1$ and with
$v=|x|^\gamma u$, consider the quotient
\[
\mathcal Q[v]:=\frac{\(\ird{|x|^{-\gamma}\,|\nabla v|^2}\right)^a\(\ird{|x|^{-\gamma}\,|v|}\right)^{2(1-a)}}{\ird{|x|^{-\gamma}\,v^2}}
\]
As a consequence of~\eqref{CKN2.0}, $\mathcal Q[v]$ is bounded from below by a positive constant if $\gamma\le0$ and $a=(d-\gamma)/(d-\gamma+2)$. Let us consider the case $\gamma>0$.
%---------------------------------------------------------------------
\begin{lemma}\phantomsection Let $d\ge1$, $\gamma\in(0,d)$ and $a=(d-\gamma)/(d-\gamma+2)$. Then there exists a sequence $(v_n)_{n\in\N}$ of smooth, compactly supported functions such that \hbox{$\displaystyle\lim_{n\to\infty}\mathcal Q[v_n]=0$}.\end{lemma}
%---------------------------------------------------------------------
\begin{proof} Let us take a smooth function $v$ and consider $v_n(x)=v(x+n\,\mathsf e)$ for some $\mathsf e\in\mathbb S^{d-1}$. Then $\mathcal Q[v_n]=O\(n^{-(1-a)\,\gamma}\)$. With $\gamma>0$, we know that $a$ is in the range $0<a<1$ if and only if $\gamma\in(0,d)$.\end{proof}

For the proof of Theorem~\ref{thm:macrow} in the case $0<\gamma<d$, $V=V_1$, we start by estimating the growth of the moment
\[
M_k(t):=\ird{|x|^k u}\,,
\]
which evolves according to
\[
M_k'=k\,\big(d+k-2-\gamma\big)\ird{u\,|x|^{k-2-\gamma}}\le k\,\big(d+k-2-\gamma\big)\,M_0^\frac2k\,M_k^{1-\frac2k}\,,
\]
where we have used H\"older's inequality and $M_0(t)=M_0(0)=\nrm{u_0}1$. Integration gives
\[
M_k(t)\le\(M_k(0)^{2/k} + 2\big(d+k-2-\gamma\big)\,M_0^{2/k}t\)^{k/2}\,.
\]
If $\gamma\in(0,d)$ and $a=\frac{d+2k-\gamma}{d+2k+2-\gamma}$, by inserting the Caffarelli-Kohn-Nirenberg inequality
\[\label{CKN2moment.0}
\ird{|x|^\gamma\,u^2}\le\mathcal C\(\ird{|x|^{-\gamma}\,\left|\nabla\(|x|^\gamma u\)\right|^2}\)^a\(\ird{|x|^k\,|u|}\)^{2(1-a)}
\]
(see Appendix~\ref{Appendix:A}, Ineq.~\eqref{CKN2moment} applied to $v=|x|^\gamma u$) in~\eqref{weighted-dissipation}, we observe that the function \hbox{$z=\ird{u^2\,|x|^\gamma}$} solves
\[\label{Ineq:Gronwall}
\frac{dz}{dt}\le-\,2\(\mathcal C^{-1}\,z\)^{1+\frac2{d+2k-\gamma}} M_k(t)^{-\frac4{d+2k-\gamma}}\,,
\]
and, after integration,
\[
z(t)\le z(0)\(1+a\(\big(1+b\,t\big)^{1-\frac{2k}{d+2k-\gamma}}-1\)\)^{-\frac{d+2k-\gamma}2}
\]
with $a$ and $b$ depend only on the quantities entering into the constant $c$ of Theorem~\ref{thm:macrow}. Let $\theta=2k/(d+2k-\gamma)$ and observe that
\[
1+a\(\big(1+b\,t\big)^{1-\theta}-1\)\ge\big(1+c\,t\big)^{1-\theta}\quad\forall\,t\ge0\,,
\]
if $c=b\,\min\left\{a,a^{1/(1-\theta)}\right\}$. Our estimate becomes
\begin{multline*}
z(t)\le z(0)\(1+a\(\big(1+b\,t\big)^{1-\theta}-1\)\)^{-k/\theta}\\
\le z(0)\big(1+c\,t\big)^{-k\,(1-\theta)/\theta}=z(0)\big(1+c\,t\big)^{-\frac{d-\gamma}2}\,.
\end{multline*}
In the case $V=V_2$ we can adopt the same strategy, based on a moment now defined as
\[
M_k(t):=\ird{\bangle x^k u}\,,
\]
and on the inhomogeneous Caffarelli-Kohn-Nirenberg inequality
\begin{multline*}
\ird{\bangle x^\gamma u^2}\le\mathcal K\(\ird{\bangle x^{-\gamma}\,|\nabla(\bangle x^\gamma u)|^2}\)^a\,M_k^{2(1-a)}\\
\mbox{with}\quad a=\frac{d+2k-\gamma}{d+2+2k-\gamma}
\end{multline*}
(see Appendix~\ref{Appendix:B}, Ineq.~\eqref{CKN2momentBis} applied to $v=\bangle x^\gamma u$). This completes the proof of Theorem~\ref{thm:macrow}.\qed

%%%%%%%%%%%%%%%%%%%%%%%%%%%%%%%%%%%%%%%%%%%%%%%%%%%%%%%%%%%%%%%%%%%%%%
%%%%%%%%%%%%%%%%%%%%%%%%%%%%%%%%%%%%%%%%%%%%%%%%%%%%%%%%%%%%%%%%%%%%%%
\subsection{Decay in self-similar variables and intermediate asymptotics}\label{Sec:Decay}~

We prove Theorem~\ref{thm:weighted}. With the parabolic change of variables
\be{ChangeOfVariables}
u(t,x)=(1+2\,t)^{-d/2}\,v(\tau,\xi)\,,\quad\tau=\tfrac12\,\log(1+2\,t)\,,
\quad\xi=\frac x{\sqrt{1+2\,t}}\,,
\ee
which preserves mass and initial data,~\eqref{eq:FPmacro} is changed into
\be{ResEqn}
\frac{\partial v}{\partial \tau}=\Delta_\xi v+\nabla_\xi\cdot (v\,\nabla_\xi\Phi)\,,
\ee
where
\[
\Phi(\tau,\xi)=V\(e^\tau\,\xi\)+\tfrac12\,|\xi|^2\,.
\]
We investigate the long-time behavior of solutions of~\eqref{eq:FPmacro} by considering quasi-equilibria
\be{quasi-equ}
v_\star(\tau,\xi):=M(\tau)\,e^{-\Phi(\tau,\xi)}\,,
\ee
of~\eqref{ResEqn} with an appropriately chosen $M(\tau)$.

For the scale invariant case $V=V_1$, the potential $\Phi_1(\tau,\xi)=\gamma\big(\log|\xi|+\tau\big)+\frac12\,|\xi|^2$ in~\eqref{ResEqn} can be replaced by the time independent potential $\phi_1(x)=\gamma\,\log|\xi|+\frac12|\xi|^2$. With $M(\tau)=c_\star\,e^{\gamma\,\tau}$, time independent equilibria
\be{v-star1}
v_{\star,1}(\xi):=c_\star\,|\xi|^{-\gamma}\,e^{-|\xi|^2/2}\,,
\ee
are available. For the second case $V=V_2$ with potential
\[
\Phi_2(\tau,\xi):=\tfrac\gamma2\,\log\(1+e^{2\tau}\,|\xi|^2\)+\tfrac12\,|\xi|^2\,,
\]
we shall use
\be{v-star2}
v_{\star,2}(\tau,\xi):=c_\star \(e^{-2\tau}+|\xi|^2\)^{-\gamma/2}\,e^{-|\xi|^2/2}\,,
\ee
so that $v_{\star,2}$ is asymptotically equivalent to $v_{\star,1}$ as $\tau\to\infty$.

\medskip If a quasi-equilibrium of the form~\eqref{quasi-equ} satisfies
\[
\frac{\partial v_\star}{\partial \tau}\ge0\,,
\]
which holds for both examples~\eqref{v-star1} and~\eqref{v-star2} if $\gamma>0$, then $v_\star$ is obviously a super-solution of~\eqref{ResEqn}, thus proving the following result on \emph{uniform decay estimates}.
%---------------------------------------------------------------------
\begin{proposition}\phantomsection\label{Prop:Uniform} Let $\gamma\in(0,d)$ and $u(t,x)$ be a solution of~\eqref{eq:FPmacro} with initial datum such that, for some constant $c_\star>0$,
\[
0\le u(0,x)\le c_\star\,\big(\sigma+|x|^2\big)^{-\gamma/2}\,\exp\(-\frac{|x|^2}2\)\quad\forall\,x\in\R^d\,,
\]
with $\sigma=0$ if $V=V_1$ and $\sigma=1$ if $V=V_2$. Then
\[
0\le u(t,x)\le\frac{c_\star}{(1+2\,t)^\frac{d-\gamma}2}\,\big(\sigma+|x|^2\big)^{-\gamma/2}\,\exp\(-\frac{|x|^2}{2\,(1+2\,t)}\)
\quad\forall\,x\in\R^d\,,\;t\ge0\,.
\]
\end{proposition}
%---------------------------------------------------------------------
For $0<\gamma<d$, we obtain a pointwise decay: the attracting potential is too weak for confinement (no stationary state can exist, at least among $\mathrm L^1(\R^d)$ solutions) but it slows down the decay compared to solutions of the heat equation (that is, solutions corresponding to $V=0$).

The result of Proposition~\ref{Prop:Uniform} is also true for $\gamma\le0$ if $V=V_1$. In that case, a repulsive potential with $\gamma<0$ accelerates the pointwise decay, but does not change the uniform decay rate as $t\to+\infty$ because
\be{Estim:Unif}
\forall\,t>0\,,\quad\max_{r>0}r^{-\gamma}\,\exp\(-\frac{r^2}{4\,t}\)=\(\frac e{2\,|\gamma|\,t}\)^{\gamma/2}\,.
\ee

\medskip In order to obtain an estimate in $\mathrm L^2\big(e^Vdx\big)$, let us state a result on a Poincar\'e inequality. We introduce the notations
\[
\Phi_{\gamma,\sigma}(\xi):=\tfrac12\,|\xi|^2+\tfrac\gamma2\,\log\(\sigma+|\xi|^2\)\,,
\]
\[
Z_{\gamma,\sigma}:=\int_{\R^d}e^{-\Phi_{\gamma,\sigma}(\xi)}\,d\xi\quad\mbox{and}\quad d\mu_{\gamma,\sigma}:=Z_{\gamma,\sigma}^{-1}\,e^{-\Phi_{\gamma,\sigma}}\,d\xi\,.
\]
%---------------------------------------------------------------------
\begin{lemma}\phantomsection\label{Lem:Poincare} Assume that $d\ge1$, $\gamma\in(0,d)$ and $\sigma\in\R^+$. With the above notations, there is a positive constant $\lambda_{\gamma,\sigma}$ such that
\begin{multline}\label{Ineq:Poincare}
\irdmuxi{|\nabla w|^2}{\gamma,\sigma}\ge\lambda_{\gamma,\sigma}\irdmuxi{|w-\overline w|^2}{\gamma,\sigma}\\
\forall\,w\in\mathrm H^1(\R^d,d\mu_{\gamma,\sigma})\;\mbox{such that}\;\overline w=\irdmuxi w{\gamma,\sigma}\,.
\end{multline}
Moreover, for any $\gamma\in(0,d)$, $\min_{\sigma\in[0,1]}\lambda_{\gamma,\sigma}>0$.
\end{lemma}
%---------------------------------------------------------------------
\begin{proof} Let us consider a potential $\psi$ on $\R^d$. We assume that $\psi$ is a measurable function such that
\[
\ell=\lim_{r\to+\infty}\mathrm\inf_{f\in\mathcal D(B_r^c)\setminus\{0\}}\frac{\irdxi{\(|\nabla f|^2+\psi\,|f|^2\)}}{\irdxi{|f|^2}}>0\,,
\]
where $B_r^c:=\left\{x\in\R^d\,:\,|x|>r\right\}$ and $\mathcal D(B_r^c)$ denotes the space of smooth functions on $\R^d$ with compact support in $B_r^c$. According to Persson's result~\cite[Theorem~2.1]{MR0133586}, the lower end of the continuous spectrum of the Schr\"odinger operator $-\,\Delta+\psi$ is $\ell$.

With $w=f\,e^{\Phi_{\gamma,\sigma}/2}$ and $\psi=\frac14\,|\nabla\Phi_{\gamma,\sigma}|^2-\frac12\,\Delta\Phi_{\gamma,\sigma}$, $\lambda_{\gamma,\sigma}$ is either $\ell=4$ if $-\,\Delta+\psi$ has no eigenvalue in the interval $(0,4)$, or the lowest positive eigenvalue of $-\,\Delta+\psi$ in the interval $(0,4)$, since the kernel is generated by constant functions. This proves that $0<\lambda_{\gamma,\sigma}\le4$. An elementary computation shows that
\[
4\,\psi(\xi)=X-\big(2\,d+\sigma-2\,\gamma\big)-\gamma\,\big(2\,d+2\,\sigma-\gamma-4\big)\,X^{-1}-\gamma\,(\gamma+4)\,X^{-2}
\]
with $X=|\xi|^2+\sigma$, which allows in principle for an explicit computation of $\lambda_{\gamma,\sigma}$ and shows that it is continuous with respect to $\sigma$ on $\R^+$.\end{proof}

In the special case $\sigma=0$, it is possible to compute $\lambda_{\gamma,0}$ as follows.
%---------------------------------------------------------------------
\begin{lemma}\phantomsection\label{Lem:Poincare0} Assume that $d\ge1$ and $\gamma\in(0,d)$. With the above notations, we have $\lambda_{\gamma,0}=4\,(1-\gamma)$ if $d=1$ and $\lambda_{\gamma,0}=\min\left\{4,\,4\,(d-\gamma),\,d-1\right\}$ if $d\ge2$.\end{lemma}
%---------------------------------------------------------------------
\begin{proof} A decomposition in spherical harmonics shows that the lowest eigenvalue associated with a non-radial eigenfunction (in dimension $d\ge2$) is of the form $f(\xi)=g(r)\,Y(\omega)$ with $r=|\xi|$, $\omega=\xi/r$ and $-\Delta_{\mathbb S^{d-1}}Y=k\,(k+d-2)\,Y$, $k\in\N$. If $k\neq0$, $g\equiv1$ is optimal and the eigenvalue is $k\,(k+d-2)$ with $k=1$. Otherwise $k=0$ and $g$ is the lowest non-trivial Hermite polynomial with zero average on $\R^+\ni r$ in dimension $n=d-\gamma$, that is $g(r)=r^2-n$ and the corresponding eigenvalue is $4n$. Notice that $n$ is not necessarily an integer, but can be considered as a real parameter. All other eigenvalues are larger. We conclude by taking the minimum of the two eigenvalues. If $d=1$, a similar conclusion holds with $f(\xi)=\xi$.\end{proof}

An interesting consequence of Lemma~\ref{Lem:Poincare0} is a result of \emph{intermediate asymptotics}, which allows to identify the leading order term of the solution of~\eqref{eq:FPmacro} as $t\to+\infty$.
%---------------------------------------------------------------------
\begin{corollary}\phantomsection\label{Cor:IA} Assume that $d\ge1$, $\gamma\in(0,d)$ and $V=V_1$. With the above notations, if $u$ solves~\eqref{eq:FPmacro} with an initial datum $u_0\in\mathrm L^1_+(\R^d)$ such that $\big(u_\star(0,x)\big)^{-1}\,u_0^2\in\mathrm L^1_+(\R^d)$, with $u_\star$ defined by~\eqref{ustar1}, and if we choose $c_\star$ in~\eqref{ustar1} such that $\|u_\star(0,\cdot)\|_1=\|u_0\|_1$, then
\[
\ird{\frac{\bigl(u(t,x)-u_\star(t,x)\bigr)^2}{u_\star(t,x)}}\le(1+2\,t)^{-\lambda_{\gamma,0}}\ird{\frac{\bigl(u(0,x)-u_\star(0,x)\bigr)^2}{u_\star(0,x)}}\,.
\]
\end{corollary}
%---------------------------------------------------------------------
\begin{proof} By definition of $u_\star$, we have
\[
\irdxi{v_{\star,1}}=\irdxi{v(0,\xi)}=\ird{u_0}\,.
\]
Then, using the Poincar\'e inequality~\eqref{Ineq:Poincare} and Lemma~\ref{Lem:Poincare0}, we know that
\begin{multline*}
\frac d{d\tau}\irdxi{(v-v_{1,\star})^2\,e^{\phi_1}}=-\,2\irdxi{\big|\nabla_\xi\big(e^{\phi_1}(v-v_{1,\star})\big)\big|^2\,e^{-\phi_1}}\\
\le-\,2\,\lambda_{\gamma,0}\irdxi{(v-v_{1,\star})^2\,e^{\phi_1}}\,,
\end{multline*}
from which we deduce that
\[
\irdxi{(v-v_{1,\star})^2\,e^{\phi_1}}\le e^{-2\,\lambda_{\gamma,0}\tau}\ird{(u(0,x)-v_{1,\star})^2\,e^{\phi_1}}\,.
\]
This concludes the proof using the parabolic change of variables\eqref{ChangeOfVariables}.\end{proof}

\begin{proof}[Proof of Theorem~\ref{thm:weighted}]
A Cauchy-Schwarz inequality shows that
\begin{multline*}
\(\ird{\bigl|u(t,x)-u_\star(t,x)\bigr|}\)^2\le\ird{u_\star(t,x)}\ird{\frac{\bigl(u(t,x)-u_\star(t,x)\bigr)^2}{u_\star(t,x)}}\\
\le(1+2\,t)^{-\lambda_{\gamma,0}}\ird{u_0}\ird{\frac{\bigl(u(0,x)-u_\star(0,x)\bigr)^2}{u_\star(0,x)}}\,.
\end{multline*}
The H\"older interpolation inequality
\[
\nrm{u(t,\cdot)-u_\star(t,\cdot)}p\le\nrm{u(t,\cdot)-u_\star(t,\cdot)}1^\frac1p\,\nrm{u(t,\cdot)-u_\star(t,\cdot)}\infty^{1-\frac1p}
\]
combined with the results of Proposition~\ref{Prop:Uniform} and Corollary~\ref{Cor:IA} concludes the proof after taking~\eqref{Estim:Unif} and the expression of $\lambda_{\gamma,0}$ stated in Lemma~\ref{Lem:Poincare0} into account.\end{proof}

%%%%%%%%%%%%%%%%%%%%%%%%%%%%%%%%%%%%%%%%%%%%%%%%%%%%%%%%%%%%%%%%%%%%%%
%%%%%%%%%%%%%%%%%%%%%%%%%%%%%%%%%%%%%%%%%%%%%%%%%%%%%%%%%%%%%%%%%%%%%%
\section{Decay estimate for the kinetic equation with weak confinement}\label{Sec:kinetic}

In this section, we prove Theorem~\ref{thm:kinw} by revisiting the $\mathrm L^2$ approach of~\cite{MR3324910} in the spirit of~\cite{BDMMS}.

%%%%%%%%%%%%%%%%%%%%%%%%%%%%%%%%%%%%%%%%%%%%%%%%%%%%%%%%%%%%%%%%%%%%%%
\subsection{Notations and elementary computations}\label{Sec:kinetic-notations}

On the space $\mathrm L^2(\mathcal M^{-1}dx\,dv)$, we define the scalar product
\[
\scalar fg=\iint_{\R^d\times\R^d}f\,g\,e^V\,M^{-1}\,dx\,dv
\]
and the norm $\|f\|=\scalar ff^{1/2}$. Let $\op\Pi$ be the orthogonal projection operator on $\text{Ker}(\op L)$ given by $\op\Pi f:=M\,\rho[f]$, where $\rho[f]:=\int_{\R^d}f(v)\,dv$, and $\op T$ be the transport operator such that $\op Tf=v\cdot\nabla_xf-\nabla_xV\cdot\nabla_vf$. We assume that
\[
M(v)=(2\pi)^{-\frac d2}\,e^{-\frac12\,|v|^2}\quad\forall\,v\in\R^d\,.
\]
Let us use the notation $u[f]:=e^V\,\rho[f]$ and observe that
\begin{multline*}
\op{T\,\Pi}f=M\,e^{-V}\,v\cdot\nabla_xu[f]\,,\quad\op{(T\,\Pi)^*}f=-\,M\,\nabla_x\cdot\rho\left[v\,f\right]\,,\\
(\op{T\Pi})^*(\op{T\Pi})f=-\,M\,\nabla_x\cdot\(e^{-V}\,\nabla_xu[f]\)\,,
\end{multline*}
where the last identity follows from $\int_{\R^d}M(v)\,v\otimes v\,dv=\op{Id}$. To build a suitable Lyapunov functional, as in~\cite{Dolbeault2009511,MR3324910,BDMMS} we introduce the operator $\op A$ defined by
\[
\op A:=\big(\op{Id}+(\op{T\Pi})^*(\op{T\Pi})\big)^{-1}(\op{T\Pi})^*\,.
\]
As in~\cite{MR3324910} we define the Lyapunov functional $\op H$ by
\[
\op H[f]:=\frac12\,\|f\|^2+\var\,\scalar{\op Af}f
\]
and obtain by a direct computation that
\[\label{E-EP}
\dt\op H[f]=-\,\op D[f]
\]
with
\begin{multline} \label{def:D}\hspace*{2cm}
\op D[f]:=-\,\scalar{\op Lf}f+\var\,\scalar{\op A\op{T\Pi} f}{\op\Pi f}+\var\,\scalar{\op A\op T(\op{Id}-\op\Pi)f}{\op\Pi f}\\
 -\,\var\,\scalar{\op{TA}(\op{Id}-\op\Pi)f}{(\op{Id}-\op\Pi)f}-\var\,\scalar{\op{AL}(\op{Id}-\op\Pi)f}f\,,
\end{multline}
where we have used that $\scalar{\op Af}{\op Lf}=0$. For the first term in $\op D[f]$, we rely on the \emph{microscopic coercivity} estimate (see~\cite{MR3324910})
\[
-\scalar{\op L f}f\ge\lambda_m\,\|(\op{Id}-\op\Pi)f\|^2\,.
\]
The second term $\scalar{\op A\op{T\Pi} f}{\op\Pi f}$ is expected to control the macroscopic contribution $\|\Pi f\|$. In Section~\ref{Sec:kinetic-elliptic} the remaining terms will be estimated to show that for $\var$ small enough $\op D[f]$ controls $\|(\op{Id}-\op\Pi)f\|^2 + \scalar{\op A\op{T\Pi} f}{\op\Pi f}$. As in Section~\ref{Sec:Theorem2}, estimates on moments are needed, which will be proved in Section~\ref{Sec:kinetic-moment} and used in Section~\ref{Sec:kinetic-Nash} to show a Nash type estimate and to complete the proof of Theorem~\ref{thm:kinw} by relating the entropy dissipation $\op D[f]$ to $\op H[f]$ and by solving the resulting differential inequality.

%%%%%%%%%%%%%%%%%%%%%%%%%%%%%%%%%%%%%%%%%%%%%%%%%%%%%%%%%%%%%%%%%%%%%%
\subsection{Proof of the Lyapunov functional property of \texorpdfstring{$\op H[f]$}{H[f]}}\label{Sec:kinetic-elliptic}

Let us define the notations
\[
\scalar{u_1}{u_2}_V:=\ird{u_1\,u_2\,e^{-V}}\quad\mbox{and}\quad\|u\|_V^2:=\scalar uu_V
\]
associated with the norm $\mathrm L^2(e^{-V}\,dx)$. Unless it is specified, $\nabla$ means $\nabla_x$.
%---------------------------------------------------------------------
\begin{lemma}\phantomsection\label{Lemma:A-TA} With the above notations, we have
\[
\|\op Af\|\le\frac12\,\|(\op{Id-\Pi})f\|\,,\kern6pt\|\op{TA}f\|\le \|(\op{Id-\Pi})f\|
\] 
and
\[
\left|\scalar{\op{TA}(\op{Id-\Pi})f}{(\op{Id-\Pi})f}\right|\le\|(\op{Id-\Pi})f\|^2\,.
\]
\end{lemma}
%---------------------------------------------------------------------
\begin{proof} We already know from~\cite[Lemma~1]{MR3324910} that the operator $\op{TA}$ is bounded. Let us give a short proof for completeness. The equation $\op Af=g$ is equivalent to
\be{fg}
(\op{T\Pi})^*f=g+(\op{T\Pi})^*\,(\op{T\Pi})\,g\,.
\ee
Multiplying~\eqref{fg} by $g\,M^{-1}\,e^V$, we get that
\begin{multline*}
\|g\|^2+\|\op{T\Pi}g\|^2=\scalar f{\op{T\Pi}g}=\scalar{(\op{Id-\Pi})f}{\op{T\Pi}g}\\
\le\|(\op{Id-\Pi})f\|\,\|\op{T\Pi}g\|\le\frac14\,\|(\op{Id-\Pi})f\|^2+\|\op{T\Pi}g\|^2
\end{multline*}
from which we deduce that $\|\op Af\|=\|g\|\le\frac12\,\|(\op{Id-\Pi})f\|$. Since $\op A=\op{\Pi A}$, because~\eqref{fg} can be rewritten as $g=\op{\Pi T^2\Pi}g-\op{\Pi T}f$ using $(\op{T\Pi})^*=-\,\op{\Pi T}$, we also have that $\op{TA}f=\op{T\Pi}g$ and obtain that $\|\op{TA}f\|=\|\op{T\Pi}g\|\le\|(\op{Id-\Pi})f\|$. The estimate on $\left|\scalar{\op{TA}(\op{Id-\Pi})f}{(\op{Id-\Pi})f}\right|$ follows.\end{proof}

The term $\scalar{\op{AT\Pi}f}{\Pi f}$ is the one which gives the macroscopic decay rate. Let $w[f]$ be such that $\big(\op{Id}+(\op{T\Pi})^*(\op{T\Pi})\big)^{-1}\op\Pi f=w\,M\,e^{-V}$. Then $w$ solves
\be{Eqn:w}
w-\mathcal Lw=u[f]\quad\mbox{where}\quad\mathcal Lw:=e^V\,\nabla\cdot\(e^{-V}\,\nabla w\)\,.
\ee
%---------------------------------------------------------------------
\begin{lemma}\phantomsection\label{Lemma:ATPi} With the above notations, if $u=u[f]$ and $w=w[f]$ solves~\eqref{Eqn:w}, we have
\[
\scalar{\op{AT\Pi}f}{\Pi f}=\|\nabla w\|_V^2+\|\mathcal Lw\|_V^2\le\frac54\,\|u\|_V^2\,.
\]
\end{lemma}
%---------------------------------------------------------------------
\begin{proof} Let $w$ be a solution of~\eqref{Eqn:w}. Since
\begin{multline*}
\op{AT\Pi}f=\big(\op{Id}+(\op{T\Pi})^*(\op{T\Pi})\big)^{-1}(\op{T\Pi})^*(\op{T\Pi})\,\op\Pi f\\
=\big(\op{Id}+(\op{T\Pi})^*(\op{T\Pi})\big)^{-1}\big(\op{Id}+(\op{T\Pi})^*\op{T\Pi}-\op{Id}\big)\,\op\Pi f\\
=\op\Pi f-\big(\op{Id}+(\op{T\Pi})^*(\op{T\Pi})\big)^{-1}\,\op\Pi f=\op\Pi f-w\,M\,e^{-V}\,,
\end{multline*}
we obtain that
\[
\op{AT\Pi}f=(u-w)\,M\,e^{-V}\,.
\]
Using~\eqref{Eqn:w} and integrating on $\R^d$ after multiplying by $\op\Pi f=u\,M\,e^{-V}$, we obtain that
\[
\scalar{\op{AT\Pi}f}{\Pi f}=\scalar u{u-w}_V=\scalar{w-\mathcal Lw}{-\mathcal Lw}_V=\|\nabla w\|_V^2+\|\mathcal Lw\|_V^2\,.
\]
On the other hand, we can also write that
\[\label{ATPi}
\scalar{\op{AT\Pi}f}{\Pi f}=\scalar u{u-w}_V=-\,\scalar u{\mathcal Lw}_V
\]
and obtain that
\[
\|\nabla w\|_V^2+\|\mathcal Lw\|_V^2=-\,\scalar u{\mathcal Lw}_V\le \|u\|_V\,\|\mathcal Lw\|_V
\le \frac14\,\|u\|_V^2+\|\mathcal Lw\|_V^2\,,
\]
using the Cauchy-Schwarz inequality. As a consequence, we obtain that
\[
\|\nabla w\|_V^2\le\frac14\,\|u\|_V^2\quad\mbox{and}\quad\|\mathcal Lw\|_V\le\|u\|_V\,,
\]
which concludes the proof.\end{proof}

%---------------------------------------------------------------------
\begin{lemma}\phantomsection\label{Lemma:Hessian} With the above notations, if $u=u[f]$ and $w$ solves~\eqref{Eqn:w}, we have
\[
\|\mathrm{Hess}(w)\|_V^2\le\max\{1,\,\gamma\}\,\scalar{\op{AT\Pi}f}{\Pi f}\,.
\]
\end{lemma}
%---------------------------------------------------------------------
\begin{proof} The operator $\mathcal L=\Delta-\nabla V\cdot\nabla$ is such that
\[
[\mathcal L,\nabla]\,w=\mathcal L(\nabla w)-\nabla(\mathcal L\,w)=\(\mathcal L\big(\tfrac{\partial w}{\partial x_i}\big)-\tfrac\partial{\partial x_i}(\mathcal L\,w)\)_{i=1}^d=-\,\mathrm{Hess}(V)\cdot\nabla w
\]
and it is self-adjoint on $\mathrm L^2(e^V\,dx)$ so that
\[
\scalar{\mathcal Lw_1}{w_2}_V=-\,\scalar{\nabla w_1}{\nabla w_2}_V=\scalar{w_1}{\mathcal Lw_2}_V
\]
for any $w_1$ and $w_2$. Applied first with $w_1=w$ and $w_2=\mathcal Lw$ and then with $w_1=w_2=\nabla w$, this shows that
\begin{multline*}
\|\mathcal Lw\|_V^2=-\,\scalar{\nabla w}{\nabla\mathcal Lw}_V=-\,\scalar{\nabla w}{\mathcal L\nabla w}_V-\ird{\nabla w\cdot[\mathcal L,\nabla]\,w\,e^{-V}}\\
=\|\mathrm{Hess}(w)\|_V^2 + \ird{\mathrm{Hess}(V):(\nabla w\otimes\nabla w)e^{-V}}
\end{multline*}
where $\|\mathrm{Hess}(w)\|_2^2=\ird{|\mathrm{Hess}(w)|^2\,e^{-V}}=\sum_{i,j=1}^d\ird{\big(\frac{\partial^2w}{\partial{x_i}\partial{x_j}}\big)^2\,e^{-V}}$. In the case \hbox{$V=V_2$}, we deduce from
\[
\frac{\partial^2V}{\partial{x_i}\partial{x_j}}=\frac\gamma{\bangle x^2}\(\delta_{ij}-2\,\frac{x_i\,x_j}{\bangle x^2}\)
\]
that
\[
\mathrm{Hess}(V)\ge-\gamma\,\op{Id}\,.
\]
Hence
\begin{multline*}
\max\{1,\,\gamma\}\,\scalar{\op{AT\Pi}f}{\Pi f} \ge \|\mathcal Lw\|_V^2 + \max\{1,\,\gamma\}\,\|\nabla w\|_V^2\\
\ge \|\mathrm{Hess}(w)\|_V^2-\gamma\,\|\nabla w\|_V^2+\max\{1,\,\gamma\}\,\|\nabla w\|_V^2\,,
\end{multline*}
which concludes the proof.\end{proof}

%---------------------------------------------------------------------
\begin{lemma}\phantomsection\label{Lemma:AT(1-Pi)} With the above notations and with $m_\gamma:=3\max\{1,\,\gamma\}$,
we have
\[
\left|\scalar{\op A\op T(\op{Id}-\op\Pi)f}{\op\Pi f}\right|\le m_\gamma\,\scalar{\op{AT\Pi}f}{\Pi f}^{1/2}\,\|(\op{Id}-\op\Pi)f\|\,.
\]
\end{lemma}
%---------------------------------------------------------------------
\begin{proof} Assume that $u=u[f]$ and $w$ solves~\eqref{Eqn:w}. Using $g=\big(\op{Id}+(\op{T\Pi})^*(\op{T\Pi})\big)^{-1}f$ so that $\big(\op{Id}+(\op{T\Pi})^*(\op{T\Pi})\big)\,g=f$ means $g-(\mathcal Lw)\,M\,e^{-V}=f$, let us compute
\begin{multline*}
\scalar{\op A\op T(\op{Id}-\op\Pi)f}{\op\Pi f}=\scalar{\op T(\op{Id}-\op\Pi)f}{\op A^*\op\Pi f}=\scalar{\op T(\op{Id}-\op\Pi)f}{\op{T\Pi}g}\\
=\iint_{\R^d\times\R^d}\sqrt M\,v\otimes v\,\frac{(\op{Id}-\op\Pi)f}{\sqrt M}\,:\,\mathrm{Hess}w\,dx\,dv\\
=\iint_{\R^d\times\R^d}\sqrt M\(v\otimes v-\tfrac1d\,\op{Id}\)\frac{(\op{Id}-\op\Pi)f}{\sqrt M}\,:\,\mathrm{Hess}w\,dx\,dv
\end{multline*}
We conclude using a Cauchy-Schwarz inequality, Lemma~\ref{Lemma:ATPi} and Lemma~\ref{Lemma:Hessian}.\end{proof}

In order to have unified notations, we adopt the convention that $\overline\sigma=1/\sqrt2$ if $\op L$ is the Fokker-Planck operator.
%---------------------------------------------------------------------
\begin{lemma}\phantomsection\label{Lemma:AL} With the above notations, we have
\[
\scalar{\op{AL}(\op{Id}-\op\Pi)f}{\Pi f}\le\sqrt2\,\overline\sigma\,\scalar{\op{AT\Pi}f}{\Pi f}^{1/2}\,\|(\op{Id}-\op\Pi)f\|\,.
\]
\end{lemma}
%---------------------------------------------------------------------
\begin{proof} We use duality to write
\[
\scalar{\op{AL}(\op{Id}-\op\Pi)f}{\Pi f}=\scalar{\op L(\op{Id}-\op\Pi)f}{h}
\]
where $h=\op A^*f=(\op{T\Pi})g$ and $g=\big(\op{Id}+(\op{T\Pi})^*(\op{T\Pi})\big)^{-1}f$ so that
\[
\big(\op{Id}+(\op{T\Pi})^*(\op{T\Pi})\big)\,g=f
\]
and $h=v\cdot\nabla w\,M\,e^{-V}$. Here $w$ solves~\eqref{Eqn:w} with $u=u[f]$.

\smallskip\noindent$\bullet$ If $\op L$ is the Fokker-Planck operator, then $\int_{\R^d}v\,\op Lf\,dv=-\,j$ and
\[
\left|\scalar{\op{AL}(\op{Id}-\op\Pi)f}f\right|=\left|\scalar j{\nabla w}_2\right|\le \|j\|_V\,\|\nabla w\|_V\le\|(\op{Id}-\op\Pi)f\|\,\|\nabla w\|_V\,.
\]
We conclude using Lemma~\ref{Lemma:ATPi} and an estimate on $j=|j|\,\mathsf e$ where $\mathsf e\in\mathbb S^{d-1}$, that goes as follows: by computing
\begin{multline*}
|j|=\left|\int_{\R^d}v\,f\,dv\right|=\left|\int_{\R^d}v\,(\op{Id}-\op\Pi)f\,dv\right|\\\le\int_{\R^d}\((\op{Id}-\op\Pi)f\,M^{-1/2}\)\(|v\cdot\mathsf e|\,M^{1/2}\)\,dv\\
\hspace*{2cm}\le\(\int_{\R^d}\big|\,\op{Id}-\op\Pi)f\,\big|^2\,M^{-1}\,dv\int_{\R^d}|v\cdot\mathsf e|^2\,M\,dv\)^\frac12\\
=\(\int_{\R^d}\big|\,\op{Id}-\op\Pi)f\,\big|^2\,M^{-1}\,dv\)^\frac12\,,
\end{multline*}
we know that
\[
\|j\,e^V\|_V^2=\int_{\R^d}|j|^2\,e^V\,dx\le\iint_{\R^d\times\R^d}\big|\,\op{Id}-\op\Pi)f\big|^2\,M^{-1}\,e^V\,dx\,dv=\|(\op{Id}-\op\Pi)f\|^2\,.
\]

\smallskip\noindent$\bullet$ If $\op L$ is the scattering operator, then
\begin{align*}
\|\op L(\op{Id-\Pi})f\|^2&\le\overline\sigma^2\int_{\R^d}\frac1M\left|\int_{\R^d}M\,M'\,\left|\frac{f'}{M'}-\frac fM\right|\,dv'\right|^2\,dv\\
&\le\overline\sigma^2\int_{\R^d}M\left|\int_{\R^d}\sqrt{M'}\,\sqrt{M'}\,\left|\frac{f'}{M'}-\frac fM\right|\,dv'\right|^2\,dv\\
&\le\overline\sigma^2\iint_{\R^d\times\R^d}M\,M'\left|\frac{f'}{M'}-\frac fM\right|^2\,dv\,dv'\le4\,\overline\sigma^2\int_{\R^d}f^2\,M^{-1}\,dv
\end{align*}
and $\|h\|=\left\|v\cdot\nabla w\,M\,e^{-V}\right|=\|\nabla w\|_2$ so that
\[
\scalar{\op{AL}(\op{Id}-\op\Pi)f}f\le\sqrt2\,\overline\sigma\,\scalar{\op{AT\Pi}f}{\Pi f}^{1/2}\,\|(\op{Id-\Pi})f\|\,.
\]
Notice that for a nonnegative function $f$, we have the improved bounds $\|\op L(\op{Id}-\op\Pi)\|\le\overline\sigma\,\|(\op{Id-\Pi})f\|$ and $\scalar{\op{AL}(\op{Id}-\op\Pi)f}{\Pi f}\le\overline\sigma\,\scalar{\op{AT\Pi}f}{\Pi f}^{1/2}\,\|(\op{Id-\Pi})f\|$.\end{proof}

Finally, we apply the results of Lemmas~\ref{Lemma:A-TA},~\ref{Lemma:AT(1-Pi)},~\ref{Lemma:AL} to the right hand side of~\eqref{def:D}:
%---------------------------------------------------------------------
\begin{lemma}\phantomsection\label{Lemma:D-control} With the above notations, we have
\[
\op D[f] \ge \lambda_\var\(\|(\op{Id-\Pi})f\|^2+\scalar{\op{AT\Pi}f}{\Pi f}\)
\]
with
\[
\lambda_\var:=\frac12\,\left(\lambda_m-\sqrt{(\lambda_m-2\,\var)^2+\var^2\,\big(m_\gamma+\sqrt{2}\,\overline\sigma\big)^2}\;\)
\]
and $\lambda_\var>0$, if $\var>0$ is small enough.\end{lemma}
%---------------------------------------------------------------------
The functional $\op H[f]$ is a Lyapunov function in the sense that $\op D[f]\ge 0$ and the equation $\op D[f]=0$ has a unique solution $f=0$.
\begin{proof}
The above mentioned Lemmas imply
\begin{multline*}
\op D[f]\ge(\lambda_m-\var)\,\|(\op{Id-\Pi})f\|^2 + \var\,\scalar{\op{AT\Pi}f}{\Pi f}\\
-\var\,\(m_\gamma+\sqrt2\,\overline\sigma\)\|(\op{Id-\Pi})f\|\,\scalar{\op{AT\Pi}f}{\Pi f}^{1/2}\,.
\end{multline*}
The Lyapunov function property is a consequence of~\eqref{Eqn:w} and Lemma~\ref{Lemma:ATPi}.
\end{proof}

%%%%%%%%%%%%%%%%%%%%%%%%%%%%%%%%%%%%%%%%%%%%%%%%%%%%%%%%%%%%%%%%%%%%%%
\subsection{Moment estimates}\label{Sec:kinetic-moment}

Let us consider the case $V=V_2$ and define the $k^{\rm{th}}$ order moments in $x$ and $v$ by
\[
J_k(t):=\|\bangle x^k\,f(t,\cdot,\cdot)\|_1 \quad\mbox{and}\quad K_k(t):=\||v|^k\,f(t,\cdot,\cdot)\|_1.
\]
Our goal is to prove estimates on $J_k$ and $K_k$. Notice that $J_0=K_0=\|f_0\|_{\mathrm L^1(\R^d\times\R^d)}$ is constant if $f$ solves~\eqref{eq:kin}.
%---------------------------------------------------------------------
\begin{lemma}\phantomsection\label{lem:estimMomentf} Let $\gamma\in(0,d)$, $k\in\N$ with $k\ge2$, $V=V_2$ and assume that $f\in\mathrm C\big(\R^+,\,\mathrm L^2(\mathcal M^{-1}dx\,dv)\big)$ is a nonnegative solution of~\eqref{eq:kin} with initial datum~$f_0$ such that $\iint_{\R^d\times\R^d}\bangle x^k\,f_0\,dx\,dv<+\infty$ and $\iint_{\R^d\times\R^d}|v|^k\,f_0\,dx\,dv<+\infty$. There exist constants $C_2,\ldots,C_k$ such that
\be{Eqn:Mfk}
J_\ell(t)\le C_\ell\,\(1+t\)^{\ell/2}\quad\mbox{and}\quad K_\ell(t)\le C_\ell\quad \forall\,t\ge0\,,\quad \ell=2,\ldots,k\,.
\ee
\end{lemma}
%---------------------------------------------------------------------
\begin{proof}
We present the proof for a Fokker-Planck operator, the case of a scattering operator follows the same steps. A direct computation shows that
\[
\frac{dK_\ell}{dt}\le \ell\,\gamma\iint_{\R^d\times\R^d}\frac{|x\cdot v|}{\bangle x^2}\,|v|^{\ell-2}\,f(t,x,v)\,dx\,dv+\ell\,(\ell+d-2)\,K_{\ell-2}-\ell\,K_\ell\,.
\]
A bound $C_\ell$ for $K_\ell$, $\ell\in\N$, follows after observing that
\[
\iint_{\R^d\times\R^d}\frac{|x\cdot v|}{\bangle x^2}\,|v|^{\ell-2}\,f(t,x,v)\,dx\,dv\le K_{\ell-1}\le K_0^{1/\ell}\,K_\ell^{1-1/\ell}
\]
and $K_{\ell-2}\le K_0^{2/k}\,K_\ell^{1-2/\ell}$ using H\"older's inequality twice.

Next, let us compute
\[
\frac{dJ_\ell}{dt}=\ell \iint_{\R^d\times\R^d}\bangle x^{\ell-2}\,x\cdot v\,f(t,x,v)\,dx\,dv=: \ell\,L_\ell\,,
\]
and
\begin{multline}\label{Ineq:2nd order}
\frac{dL_\ell}{dt} = \iint_{\R^d\times\R^d}\bangle x^{\ell-2}\,|v|^2\,f\,dx\,dv+ (\ell-2)\iint_{\R^d\times\R^d}\bangle x^{\ell-4}\,(x\cdot v)^2\,f\,dx\,dv\\
-\gamma \iint_{\R^d\times\R^d}\bangle x^{\ell-4}\,|x|^2\,f\,dx\,dv-L_\ell\\
\le (\ell-1)\iint_{\R^d\times\R^d}\bangle x^{\ell-2}\,|v|^2\,f\,dx\,dv-L_\ell\,.
\end{multline}
Note that, again by H\"older's inequality, $|L_\ell|\le J_\ell^{1-1/\ell} K_\ell^{1/\ell}$, $\ell=2,\ldots,k$.

We prove the bound on $J_\ell(t)$ by induction. If $\ell=2$,~\eqref{Ineq:2nd order} implies $L_2(t)\le\max\big\{L_2(0),C_2\big\}$ and, thus, $J_2(t)\le C_2\,(1+t)$, up to a redefinition of $C_2$.

Now let $\ell>2$ and assume that
\[
J_{\ell-1}(t)\le C_{\ell-1}\,(1+t)^{\frac{\ell-1}{2}}\,.
\]
We use H\"older's inequality once more for the right hand side of~\eqref{Ineq:2nd order}:
\[
\frac{dL_\ell}{dt}\le(\ell-1)\,J_{\ell-1}^{\frac{\ell-2}{\ell-1}}\,K_{2(\ell-1)}^{\frac{1}{\ell-1}}-L_\ell\le (\ell-1)\,C_{\ell-1}^{\frac{\ell-2}{\ell-1}}\,C_{2(\ell-1)}^{\frac{1}{\ell-1}}\,(1+t)^{\frac{\ell}{2}-1}-L_\ell\,,
\]
which implies
\[
L_\ell\le C\,\(1+t\)^{\frac \ell2-1}\,,
\]
and one more integration with respect to $t$ establishes the estimate for $J_\ell$ in~\eqref{Eqn:Mfk}, up to an eventual redefinition of $C_\ell$.\end{proof}
%---------------------------------------------------------------------
\begin{lemma}\phantomsection\label{lem:estimMoment} Let $\gamma\in(0,d)$, $k\in\N$ with $k>2$, $V=V_2$ and assume that $f\in\mathrm C\big(\R^+,\,\mathrm L^2(\mathcal M^{-1}dx\,dv)\big)$ is a nonnegative solution of~\eqref{eq:kin} with initial datum $f_0$ such that $\iint_{\R^d\times\R^d}\bangle x^k\,f_0\,dx\,dv<+\infty$ and $\iint_{\R^d\times\R^d}|v|^k\,f_0\,dx\,dv<+\infty$. Let $w=w[f]$ be determined by~\eqref{Eqn:w} in terms of $u=u[f]$. Then there exists a positive constant $C_k$ such that
\[\label{Eqn:Mk}
0\le M_k(t):=\ird{w\,\bangle x^{k-\gamma}}\le C_k\,(1+t)^{k/2}\quad\forall\, t\ge 0\,.
\]
\end{lemma}
%---------------------------------------------------------------------
\begin{proof} The solution $w$ of~\eqref{Eqn:w} is positive by the maximum principle. In what follows we use the definition of $M_\ell$ for arbitrary integers $\ell$ and note that for $\ell\le 0$,
\be{Eqn:mass}
M_\ell\le M_0=\int_{\R^d}w\,e^{-V}\,dx=\int_{\R^d}u\,e^{-V}\,dx=\|f_0\|_1.
\ee
Multiplication of~\eqref{Eqn:w} by $\bangle x^{\ell-\gamma}$ and integration over $\R^d$ gives
\be{M-recursion}
M_\ell=\ell\,(\ell-2+d-\gamma)\,M_{\ell-2}-(\ell-2-\gamma)\,M_{\ell-4} + J_\ell\,,
\ee
where $J_\ell$ has been estimated in Lemma~\ref{lem:estimMomentf}. Then, with $\ell=2$ and~\eqref{Eqn:mass}, we obtain $M_2(t)\le C_2\,(1+t)$. This implies by the H\"older inequality that $M_1(t)\le\sqrt{M_0\,M_2(t)}\le C_1\,(1+t)^{1/2}$. For $2<\ell\le k$ the estimate $M_\ell(t)\le C_\ell\,(1+t)^{\ell/2}$ follows recursively from~\eqref{M-recursion}.
\end{proof}

%%%%%%%%%%%%%%%%%%%%%%%%%%%%%%%%%%%%%%%%%%%%%%%%%%%%%%%%%%%%%%%%%%%%%%
\subsection{Decay estimate for the kinetic equation (proof of Theorem~\ref{thm:kinw})}\label{Sec:kinetic-Nash}

%---------------------------------------------------------------------
\begin{lemma}\phantomsection\label{Prop:KineticNash} Let $\gamma\in(0,d)$, $k\ge\max\{2, \gamma/2\}$, $V=V_2$ and assume that $f\in\mathrm C\big(\R^+,\,\mathrm L^2(\mathcal M^{-1}dx\,dv)\big)$ is a nonnegative solution of~\eqref{eq:kin} with initial datum $f_0$ such that $\iint_{\R^d\times\R^d}\bangle x^k\,f_0\,dx\,dv<+\infty$ and $\iint_{\R^d\times\R^d}|v|^k\,f_0\,dx\,dv<+\infty$. Assume the above notations, in particular with $M_k$ defined as in Lemma~\ref{lem:estimMoment}, with the constant~$\mathcal K$ from~\eqref{CKN2momentBis}, and with $a=\frac{d+2k-\gamma}{d+2+2k-\gamma}$. Then
\[
\|\Pi f\|^2\le2\scalar{\op{AT\Pi}f}{\Pi f} + \mathcal K\,M_k^{2(1-a)}\scalar{\op{AT\Pi}f}{\Pi f}^a=: \Phi\left(\scalar{\op{AT\Pi}f}{\Pi f}; M_k\right)
\quad\forall\,t\ge0\,.
\]
\end{lemma}
%---------------------------------------------------------------------
\begin{proof}
If $u=u[f]$ and $w$ solves~\eqref{Eqn:w}, we recall that
\[
\scalar{\op{AT\Pi}f}{\Pi f}=\|\nabla w\|_V^2 + \|\mathcal Lw\|_V^2
\]
by Lemma~\ref{Lemma:ATPi}. From~\eqref{Eqn:w}, we also deduce that
\[
\|u\|_V^2=\scalar u{w-\mathcal Lw}_V\le\|u\|_V\(\|w\|_V^2+2\,\|\nabla w\|_V^2+\|\mathcal Lw\|_V^2\)^{1/2}\,.
\]
By inequality~\eqref{CKN2momentBis}, we have that
\[
\|w\|_V^2\le\mathcal K\,\|\nabla w\|_V^{2a}\,M_k^{2(1-a)}\,.
\]
Combining these inequalities gives
\[
\|u\|_V^2\le\mathcal K\,\|\nabla w\|_V^{2a}\,M_k^{2(1-a)} + \|\nabla w\|_V^2 + \scalar{\op{AT\Pi}f}{\Pi f}\,,
\]
which, noting that $\|\Pi f\|=\|u\|_V$, implies the result.
\end{proof}

As a consequence of Lemmas~\ref{Lemma:A-TA},~\ref{Lemma:D-control},~\ref{Prop:KineticNash} and of the properties of $\Phi$ we have
\begin{align*}
\op H[f] &= \frac12\,\|f\|^2+\var\,\scalar{\op Af}f\\&\le\frac{1+\var}2 \|f\|^2\le\frac{1+\var}2\,\left( \|(\op{Id}-\Pi)f\|^2 + \Phi\left(\scalar{\op{AT\Pi}f}{\Pi f}; M_k\right) \right) \\
&\le \frac{1+\var}2\, \Phi\left(\|(\op{Id}-\Pi)f\|^2 + \scalar{\op{AT\Pi}f}{\Pi f}; M_k\right)\le \frac{1+\var}2\, \Phi\left( \frac{\op D[f]}{\lambda_\eps}; M_k\right)\,,
\end{align*}
implying, with Lemma~\ref{lem:estimMoment},
\[
\frac{d\op H[f]}{dt}=-\,\op D[f]\le-\,\lambda_\var\,\Phi^{-1}\left( \frac{2}{1+\var} \op H[f]; C_k(1+t)^{k/2}\right)\,.
\]
The decay of $\op H[f]$ can be estimated by the solution $z$ of the corresponding ODE problem
\[
\frac{dz}{dt}=-\,\lambda_\var\,\Phi^{-1}\left( \frac{2}{1+\var} z; C_k(1+t)^{k/2}\right)\,,\quad z(0)=\op H[f_0]\,.
\]
By the properties of $\Phi$ it is obvious that $z(t)\to 0$ monotonically as $t\to+\infty$, which implies that the same is true for $\frac{dz}{dt}$. Therefore, there exists $t_0>0$ such that, in the rewritten ODE
\[
-\frac{2}{\lambda_\var} \frac{dz}{dt} + \mathcal K\,C_k^{2(1-a)}(1+t)^{k(1-a)}\left(- \frac{1}{\lambda_\var} \frac{dz}{dt}\right)^a=\frac{2z}{1+\var}\,,
\]
the first term is smaller than the second for $t\ge t_0$, implying the differential inequality
\[
\frac{dz}{dt}\le-\,\kappa\,z^{1/a} (1+t)^{k(1-1/a)}\quad\mbox{for } t\ge t_0\,,
\]
with an appropriately defined positive constant $\kappa$. Integration and estimation as in Section~\ref{Sec:Theorem2} gives
\[
z(t)\le C\,(1+t)^{\frac{1+k(1-1/a)}{1-1/a}}=C\,(1+t)^{\frac{\gamma-d}{2}}\,,
\]
thus completing the proof of Theorem~\ref{thm:kinw}.

%%%%%%%%%%%%%%%%%%%%%%%%%%%%%%%%%%%%%%%%%%%%%%%%%%%%%%%%%%%%%%%%%%%%%%
%%%%%%%%%%%%%%%%%%%%%%%%%%%%%%%%%%%%%%%%%%%%%%%%%%%%%%%%%%%%%%%%%%%%%%
\begin{center}\rule{2cm}{0.5pt}\end{center}\appendix
\section{Homogeneous Caffarelli-Kohn-Nirenberg inequalities of Nash type}\label{Appendix:A}

%%%%%%%%%%%%%%%%%%%%%%%%%%%%%%%%%%%%%%%%%%%%%%%%%%%%%%%%%%%%%%%%%%%%%%
\subsection{The general Caffarelli-Kohn-Nirenberg inequalities}

The main result of~\cite{Caffarelli-Kohn-Nirenberg-84} goes as follows. \emph{Assume that $p\ge 1$, $q\ge1$, $r>0$, $0\le a\le1$ and
\[
\frac1p+\frac\aalpha d>0\,,\quad\frac1q+\frac\bbeta d>0\,,\quad\frac1r+\frac\ggamma d>0\,,
\]
\[
\frac1r+\frac\ggamma d=a\(\frac1p+\frac{\aalpha-1}d\)+(1-a)\(\frac1q+\frac\bbeta d\)
\]
and, with $\sigma$ such that $\ggamma=a\,\sigma+(1-a)\bbeta$,
\[
0\le\aalpha-\sigma\quad\mbox{if}\quad a>0\,.
\]
Assume moreover that
\[
\aalpha-\sigma\le1\quad\mbox{if}\quad a>0\quad\mbox{and}\quad\frac1p+\frac{\aalpha-1}d=\frac1q+\frac\bbeta d\,.
\]
Then there exists a positive constant $\mathcal C$ such that the inequality
\be{CKN84}
\nrm{|x|^\ggamma\,v}r\le\mathcal C\,\nrm{|x|^\aalpha\,\nabla v}p^a\,\nrm{|x|^\bbeta\,v}q^{1-a}
\ee
holds for any $v\in C_0^\infty(\R^d)$.}

These interpolation inequalities are known in the literature as the \emph{Caffarelli-Kohn-Nirenberg inequalities} according to~\cite{Caffarelli-Kohn-Nirenberg-84} but were introduced earlier by V.P. Il'in in~\cite{Ilyin}. Next we specialize Ineq.~\eqref{CKN84} to various cases of Nash type corresponding to $q=1$.

%%%%%%%%%%%%%%%%%%%%%%%%%%%%%%%%%%%%%%%%%%%%%%%%%%%%%%%%%%%%%%%%%%%%%%
\subsection{Weighted Nash type inequalities} We consider special cases corresponding to $r=p=2$ and $q=1$.

\noindent $\bullet$ Ineq.~\eqref{CKN84} with $\aalpha=\beta/2$, $\bbeta=\beta/2$, and $\ggamma=\beta/2$ can be written under the condition $\beta>-d$ as
\begin{multline}\label{CKN1}
\ird{|x|^{\beta}\,v^2}\le\mathcal C\(\ird{|x|^{\beta}\,|\nabla v|^2}\)^a\(\ird{|x|^{\beta/2}\,|v|}\)^{2(1-a)}\\
\mbox{with}\quad a=\frac d{d+2}\,.
\end{multline}
We can indeed check that $\aalpha-\sigma=0$ for any $\beta\le0$ and $\frac1p+\frac\aalpha d>0$, $\frac1q+\frac\bbeta d>0$, and $\frac1r+\frac\ggamma d>0$ if and only if $\beta>-d$.

\noindent $\bullet$ Ineq.~\eqref{CKN84} with $\aalpha=-\,\gamma/2$, $\bbeta=k-\gamma$ and $\ggamma=-\,\gamma/2$ can be written as
\begin{multline}\label{CKN2moment}
\ird{|x|^{-\gamma}\,v^2}\le\mathcal C\(\ird{|x|^{-\gamma}\,|\nabla v|^2}\)^a\(\ird{|x|^{k-\gamma}\,|v|}\)^{2(1-a)}\\
\mbox{with}\quad a=\frac{d+2k-\gamma}{d+2+2k-\gamma}
\end{multline}
under the condition $\gamma<d$ and $k\ge\gamma/2$. We can indeed check that $\aalpha-\sigma=\frac{2k-\gamma}{2k-\gamma+d}\ge0$. In that case, we have $\aalpha-\sigma<1$ for any $\gamma\le0$ and the conditions $\frac1p+\frac\aalpha d>0$, $\frac1q+\frac\bbeta d>0$, and $\frac1r+\frac\ggamma d>0$ are always satisfied.

%%%%%%%%%%%%%%%%%%%%%%%%%%%%%%%%%%%%%%%%%%%%%%%%%%%%%%%%%%%%%%%%%%%%%%
\subsection{A weighted Nash inequality on balls}

We adapt the proof of E.~Carlen and M.~Loss in~\cite{MR1230297} to the case of homogeneous weights. With $g=|x|^{-\gamma/2}v$,~\eqref{CKN2moment} is equivalent~to
\begin{multline*}
\ird{g^2}\le\mathcal C\(\ird{|\nabla g|^2}-\frac\gamma4\,(2\,d-\gamma-4)\ird{\frac{g^2}{|x|^2}}\)^a\\
\cdot\(\ird{|x|^{k-\frac\gamma2}|g|}\)^{2(1-a)}\,.
\end{multline*}
Without loss of generality, we can assume that the function $g$ is nonnegative and radial, by spherically non-increasing rearrangements. From now on, we will only consider nonnegative, radial, non-increasing functions $g$ and the corresponding functions $v(x)=|x|^{\gamma/2}g(x)$. For any $R>0$, let
\[
g_R:=g\,\mathbbm 1_{B_R}\quad\mbox{and}\quad v_R(x)=|x|^{\gamma/2}g_R(x)\,.
\]
We observe that $g-g_R$ is supported in $\R^d\setminus B_R$ and
\[
g-g_R\le g(R)\le\overline g_R:=\frac{\ird{g_R\,|x|^{k-\gamma/2}}}{\int_{B_R}|x|^{k-\gamma/2}\,dx}=\frac{\ird{v_R\,|x|^{k-\gamma}}}{\int_{B_R}|x|^{k-\gamma/2}\,dx}
\]
because $g$ is radial non-increasing, so that
\begin{multline*}
\ird{|v-v_R|^2\,|x|^{-\gamma}}=\nrm{g-g_R}2^2\\
\le\overline g_R\ird{|g-g_R|}=\overline g_R\,R^{\frac\gamma2-k}\ird{|v-v_R|\,|x|^{k-\gamma}}\,,
\end{multline*}
\emph{i.e.},
\be{HNhom:inCKNId1}
\ird{|v-v_R|^2\,|x|^{-\gamma}}\le\frac{\ird{v_R\,|x|^{k-\gamma}}}{\int_{B_R}|x|^{k-\gamma/2}\,dx}\,R^{\frac\gamma2-k}\ird{|v-v_R|\,|x|^{k-\gamma}}\,.
\ee
On the other hand, let us define $\overline v_R:=\frac{\ird{v_R\,|x|^{k-\gamma}}}{\int_{B_R}|x|^{2k-\gamma}\,dx}$ and observe that
\[
\ird{|v_R|^2\,|x|^{k-\gamma}}=\ird{\left|v_R-\,\overline v_R\,|x|^k\right|^2\,|x|^{-\gamma}}+\overline v_R^2\int_{B_R}|x|^{2k-\gamma}\,dx\,.
\]
Let us consider the weighted inequality
\begin{multline}\label{HNhom:ineqw}
\int_{B_R}|w|^2\,|x|^{-\gamma}\,dx\le\frac1{\lambda_1^R}\int_{B_R}|\nabla w|^2\,|x|^{-\gamma}\,dx\quad\forall\,w\in\mathrm H^1(B_R,|x|^{-\gamma}\,dx)\\
\mbox{such that}\quad\int_{B_R}w\,|x|^{k-\gamma}\,dx=0\,.
\end{multline}
The existence of a positive, finite constant $\lambda_1^R$ can be deduced from elementary variational techniques as in~\cite{MR3713540}. We infer from the definition of $\overline v_R$ that this inequality is equivalent to
\be{HNhom:inCKNId2}
\ird{|v_R|^2\,|x|^{-\gamma}}\le\frac1{\lambda_1^R}\,\int_{B_R}|\nabla v|^2\,|x|^{-\gamma}\,dx+\frac{\(\ird{v_R\,|x|^{k-\gamma}}\)^2}{\int_{B_R}|x|^{2k-\gamma}\,dx}\,.
\ee
With $\lambda_1:=\lambda_1^1$, a simple scaling shows that $\lambda_1^R=\lambda_1\,R^{-2}$.

Let us come back to the estimation of $\ird{v^2\,|x|^{-\gamma}}$. By definition of $v_R$, we know that
\[
\ird{v^2\,|x|^{-\gamma}}\le\ird{|v_R|^2\,|x|^{-\gamma}}+\ird{|v-v_R|^2\,|x|^{-\gamma}}\,.
\]
After summing~\eqref{HNhom:inCKNId1} and~\eqref{HNhom:inCKNId2}, we arrive at
\begin{multline*}
\ird{v^2\,|x|^{-\gamma}}\le\frac{R^2}{\lambda_1}\,\int_{B_R}|\nabla v|^2\,|x|^{-\gamma}\,dx+\frac{\(\ird{v_R\,|x|^{k-\gamma}}\)^2}{\int_{B_R}|x|^{2k-\gamma}\,dx}\\
+\frac{R^{\frac\gamma2-k}}{\int_{B_R}|x|^{k-\gamma/2}\,dx}\ird{v_R\,|x|^{k-\gamma}}\ird{|v-v_R|\,|x|^{k-\gamma}}
\end{multline*}
and notice that
\begin{align*}
&\hspace*{-12pt}\frac{\(\ird{v_R\,|x|^{k-\gamma}}\)^2}{\int_{B_R}|x|^{2k-\gamma}\,dx}+\frac{R^{\frac\gamma2-k}}{\int_{B_R}|x|^{k-\gamma/2}\,dx}\ird{v_R\,|x|^{k-\gamma}}\ird{|v-v_R|\,|x|^{k-\gamma}}\\
&\le\ird{v_R\,|x|^{k-\gamma}}\left[\frac{\ird{v_R\,|x|^{k-\gamma}}}{\int_{B_R}|x|^{2k-\gamma}\,dx}+\frac{R^{\frac\gamma2-k}}{\int_{B_R}|x|^{k-\gamma/2}\,dx}\ird{|v-v_R|\,|x|^{k-\gamma}}\right]\\
&\le\(\ird{v\,|x|^{k-\gamma}}\)^2\,\max\left\{\frac1{\int_{B_R}|x|^{2k-\gamma}\,dx},\frac{R^{\frac\gamma2-k}}{\int_{B_R}|x|^{k-\gamma/2}\,dx}\right\}\\
&\hspace*{12pt}=\(\ird{v\,|x|^{k-\gamma}}\)^2\,\mathsf c\,R^{\gamma-d-2k}
\end{align*}
using $k>0$ and $v_R\le v$, for some numerical constant $\mathsf c$ which depends only on $d$ and $\gamma$. Collecting terms, we have found that
\[
\ird{v^2\,|x|^{-\gamma}}\le\frac{R^2}{\lambda_1}\,\int_{B_R}|\nabla v|^2\,|x|^{-\gamma}\,dx+\mathsf c\,R^{\gamma-d-2k}\(\ird{v\,|x|^{k-\gamma}}\)^2\,.
\]
We can summarize our observations as follows.
%---------------------------------------------------------------
\begin{proposition}\phantomsection\label{Prop:HNhom} Let $d\ge3$, $\gamma\in(0,d)$, $k\ge\gamma/2$ and $a=\frac{d+2k-\gamma}{d+2+2k-\gamma}$. If $\mathcal C$ denotes the optimal constant in~\eqref{CKN2moment}, then~\eqref{HNhom:ineqw} holds with a constant $\lambda_1^R=\lambda_1\,R^{-2}$ for any $R>0$, where $\lambda_1$ is a positive constant such that $\lambda_1\le\kappa\,\mathcal C^{-1/a}$ for some explicit positive constant $\kappa$ depending only on $\gamma$ and $d$.\end{proposition}
%---------------------------------------------------------------
The numerical value of $\kappa$ can be deduced from the expression of $\mathsf c$ and from the coefficients that arise from the optimization with respect to $R>0$.

%%%%%%%%%%%%%%%%%%%%%%%%%%%%%%%%%%%%%%%%%%%%%%%%%%%%%%%%%%%%%%%%%%%%%%
%%%%%%%%%%%%%%%%%%%%%%%%%%%%%%%%%%%%%%%%%%%%%%%%%%%%%%%%%%%%%%%%%%%%%%
\begin{center}\rule{2cm}{0.5pt}\end{center}
\section{Inhomogeneous Caffarelli-Kohn-Nirenberg inequalities of Nash type}\label{Appendix:B}

Our goal is to establish an extension of~\eqref{CKN2moment} adapted to the inhomogeneous case.
%---------------------------------------------------------------
\begin{theorem}\phantomsection\label{Thm:alphaStar} If $d\ge3$, $\gamma\in(0,d)$ and $k\ge\gamma/2$, then
\be{CKN2momentBis}
\ird{\bangle x^{-\gamma}v^2}\le\mathcal K\(\ird{\bangle x^{-\gamma}|\nabla v|^2}\)^a\(\ird{\bangle x^{k-\gamma}\,|v|}\)^{2(1-a)}
\ee
with $a=\frac{d+2k-\gamma}{d+2+2k-\gamma}$ holds for some optimal constant $\mathcal K>0$.\end{theorem}
%---------------------------------------------------------------
\begin{proof} Again we rely on the method of E.~Carlen and M.~Loss in~\cite{MR1230297}. The computations are similar to the ones of Proposition~\ref{Prop:HNhom} except that $|x|$ has to be replaced by $\bangle x$. With $g=\bangle x^{-\gamma/2}v$,~\eqref{CKN2momentBis} is equivalent~to
\begin{multline*}
\ird{g^2}\le\mathcal K\(\ird{|\nabla g|^2}-\frac\gamma4\,(2\,d-\gamma-4)\ird{g^2\bangle x^{-2}}\right.\\
\left.-\frac\gamma4\,(\gamma+4)\ird{g^2\bangle x^{-4}}\)^a\cdot\(\ird{\bangle x^{k-\frac\gamma2}|g|}\)^{2(1-a)}\,.
\end{multline*}
Without loss of generality, we assume that the function $g$ is nonnegative, radial by spherically non-increasing rearrangements, and nonnegative. Let $v(x)=\bangle x^{\gamma/2}g(x)$ and
\[
g_R:=g\,\mathbbm 1_{B_R}\quad\mbox{and}\quad v_R(x)=\bangle x^{\gamma/2}g_R(x)
\]
for any $R>0$. We observe that $g-g_R$ is supported in $\R^d\setminus B_R$ and
\[
g-g_R\le g(R)\le\overline g_R:=\frac{\ird{g_R\,\bangle x^{k-\gamma/2}}}{\int_{B_R}\bangle x^{k-\gamma/2}\,dx}=\frac{\ird{v_R\,\bangle x^{k-\gamma}}}{\int_{B_R}\bangle x^{k-\gamma/2}\,dx}
\]
because $g$ is radial non-increasing, so that
\begin{multline*}
\ird{|v-v_R|^2\,\bangle x^{-\gamma}}=\nrm{g-g_R}2^2\\
\le\overline g_R\ird{|g-g_R|}=\overline g_R\,\bangle R^{\frac\gamma2-k}\ird{|v-v_R|\,\bangle x^{k-\gamma}}\,,
\end{multline*}
that is,
\be{HN:inCKNId1}
\ird{|v-v_R|^2\,\bangle x^{-\gamma}}\le\frac{\ird{v_R\,\bangle x^{k-\gamma}}}{\int_{B_R}\bangle x^{k-\gamma/2}\,dx}\,\bangle R^{\frac\gamma2-k}\ird{|v-v_R|\,\bangle x^{k-\gamma}}\,.
\ee
On the other hand, using
\begin{multline*}
\ird{|v_R|^2\,\bangle x^{-\gamma}}=\ird{\left|v_R-\overline v_R\,\bangle x^k\right|^2\,\bangle x^{-\gamma}}+\overline v_R^2\int_{B_R}\bangle x^{2k-\gamma}\,dx\\
\mbox{where}\quad\overline v_R:=\frac{\ird{v_R\,\bangle x^{k-\gamma}}}{\int_{B_R}\bangle x^{2k-\gamma}\,dx}\,,
\end{multline*}
we deduce from the weighted Poincar\'e inequality
\begin{multline*}
\int_{B_R}|w|^2\,\bangle x^{-\gamma}\,dx\le\frac1{\lambda_1^R}\int_{B_R}|\nabla w|^2\,\bangle x^{-\gamma}\,dx\\
\forall\,w\in\mathrm H^1(B_R)\quad\mbox{such that}\quad\int_{B_R}w\,\bangle x^{2k-\gamma}\,dx=0
\end{multline*}
and from the definition of $\overline v_R$ that
\be{HN:inCKNId2}
\ird{|v_R|^2\,\bangle x^{-\gamma}}\le\frac1{\lambda_1^R}\,\int_{B_R}|\nabla v|^2\,\bangle x^{-\gamma}\,dx+\frac{\(\ird{v_R\,\bangle x^{k-\gamma}}\)^2}{\int_{B_R}\bangle x^{2k-\gamma}\,dx}\,.
\ee
By definition of $v_R$, we also know that
\[
\ird{v^2\,\bangle x^{-\gamma}}\le\ird{|v_R|^2\,\bangle x^{-\gamma}}+\ird{|v-v_R|^2\,\bangle x^{-\gamma}}\,.
\]
After summing~\eqref{HN:inCKNId1} and~\eqref{HN:inCKNId2}, we arrive at
\begin{multline*}
\ird{v^2\,\bangle x^{-\gamma}}\le\frac1{\lambda_1^R}\,\int_{B_R}|\nabla v|^2\,\bangle x^{-\gamma}\,dx+\frac{\(\ird{v_R\,\bangle x^{k-\gamma}}\)^2}{\int_{B_R}\bangle x^{2k-\gamma}\,dx}\\ \hspace*{2.5cm}
+\frac{\bangle R^{\frac\gamma2-k}}{\int_{B_R}\bangle x^{k-\gamma/2}\,dx}\ird{v_R\,\bangle x^{k-\gamma}}\ird{|v-v_R|\,\bangle x^{k-\gamma}}\\
\le\mathsf a(R)\int_{B_R}|\nabla v|^2\,\bangle x^{-\gamma}\,dx+\mathsf b(R)\(\ird{v\,\bangle x^{k-\gamma}}\)^2
\end{multline*}
where $\mathsf a$ and $\mathsf b$ are two positive continuous functions on $(0,+\infty)$ defined by
\[
\mathsf a(R):=1/\lambda_1^R\quad\mbox{and}\quad\mathsf b(R):=\max\left\{\frac1{\int_{B_R}\bangle x^{2k-\gamma}\,dx},\frac{\bangle R^{\frac\gamma2-k}}{\int_{B_R}\bangle x^{k-\gamma/2}\,dx}\right\}
\]
and such that $\lim_{R\to0_+}R^d\,\mathsf b(R)\in(0,+\infty)$, $\lim_{R\to+\infty}R^{d+2k-\gamma}\,\mathsf b(R)\in(0,+\infty)$, $\lim_{R\to+\infty}R^{-2}\,\mathsf a(R)=1/\lambda_1$ where~$\lambda_1$ is the optimal constant in Proposition~\ref{Prop:HNhom} while $\lim_{R\to0_+}R^{-2}\,\mathsf a(R)=1/\lambda$ is related with Nash's inequality as in~\cite{MR1230297} and such that
\[
\int_{B_1}|w|^2\,dx\le\frac1\lambda\int_{B_1}|\nabla w|^2\,dx\quad\forall\,w\in\mathrm H^1(B_1)\quad\mbox{such that}\quad\int_{B_1}w\,dx=0\,.
\]

In order to prove~\eqref{CKN2momentBis}, we can use the homogeneity of the inequality and assume that $\ird{\bangle x^{-\gamma}v^2}=1$. What we shown so far is that
\[
\forall\,R>0\,,\quad1\le\(\ird{\bangle x^{k-\gamma}\,|v|}\)^2\,\big(\mathsf a(R)\,X+\mathsf b(R)\big)
\]
where $X=\ird{\!\bangle x^{-\gamma}\,|\nabla v|^2}/\!\(\ird{\!\bangle x^{k-\gamma}\,|v|}\)^2$. With the choice $R=X^{-(1-a)/2}$, we get that there exists a constant $\mathsf K>0$ such that $\mathsf a(R)\,X+\mathsf b(R)<\mathsf K\,X^a$. This proves~\eqref{CKN2momentBis} with $\mathcal K\le\mathsf K$.
\end{proof}

%%%%%%%%%%%%%%%%%%%%%%%%%%%%%%%%%%%%%%%%%%%%%%%%%%%%%%%%%%%%%%%%%%%%%%
%%%%%%%%%%%%%%%%%%%%%%%%%%%%%%%%%%%%%%%%%%%%%%%%%%%%%%%%%%%%%%%%%%%%%%
\begin{center}\rule{2cm}{0.5pt}\end{center}
\section{Hardy-Nash inequalities}\label{Appendix:C}

%%%%%%%%%%%%%%%%%%%%%%%%%%%%%%%%%%%%%%%%%%%%%%%%%%%%%%%%%%%%%%%%%%%%%%
\subsection{Proof of Lemma~\texorpdfstring{\ref{lem:HN}}{lem:HN}}

We start with the proof of~\eqref{Ineq:Hardy-Nash2} by first showing a Hardy type inequality. For some $\alpha\in\R$ to be fixed later we compute
\begin{multline*}
0\le\ird{\left|\nabla u+\frac{\alpha\,x}{1+|x|^2}\,u\right|^2}\\
=\ird{|\nabla u|^2}+\alpha^2\ird{\frac{|x|^2\,u^2}{\(1+|x|^2\)^2}}+\alpha\ird{\nabla\(u^2\)\cdot\frac x{1+|x|^2}}\,.
\end{multline*}

We deduce that
\[
\ird{|\nabla u|^2}+\alpha^2\ird{\frac{|x|^2\,u^2}{\(1+|x|^2\)^2}}-\alpha\,d\ird{\frac{u^2}{1+|x|^2}}+2\,\alpha\ird{\frac{|x|^2\,u^2}{\(1+|x|^2\)^2}}\ge0\,,
\]
so that, by writing $|x|^2=\langle x\rangle^2-1$, we obtain
\be{InhomHardy}
\nrm{\nabla u}2^2+\alpha\,(\alpha-d+2)\ird{\frac{u^2}{1+|x|^2}}-\alpha\,(\alpha+2)\ird{\frac{u^2}{\(1+|x|^2\)^2}}\ge0\,.
\ee
Concerning the second term, we choose the optimal value $\alpha=(d-2)/2$ in~\eqref{InhomHardy}, producing the optimal upper bound for $\delta$. It is now straightforward to show
\begin{multline*}
\nrm{\nabla u}2^2-\delta\ird{\frac{u^2}{1+|x|^2}}-\eta\ird{\frac{u^2}{\(1+|x|^2\)^2}}\\
\ge\min\left\{1-\tfrac{4\,\delta}{(d-2)^2},1-\tfrac{4\,\eta}{d^2-4}\right\}\,\nrm{\nabla u}2^2\,,
\end{multline*}
whence the proof of~\eqref{Ineq:Hardy-Nash2} is completed by an application of Nash's inequality~\eqref{Ineq:Nash}.

The result~\eqref{Ineq:Hardy-Nash} is shown analogously by using the standard Hardy inequality
\be{eq:Hardy}
\nrm{\nabla u}2^2-\frac14\,(d-2)^2\ird{\frac{u^2}{|x|^2}}\ge0
\ee
instead of~\eqref{InhomHardy}. This completes the proof of Lemma~\ref{lem:HN}.\qed

%%%%%%%%%%%%%%%%%%%%%%%%%%%%%%%%%%%%%%%%%%%%%%%%%%%%%%%%%%%%%%%%%%%%%%
\subsection{Hardy-Nash vs.~Caffarelli-Kohn-Nirenberg inequalities}\label{Sec:Hardy-Nash}

The values for $\mathcal{C}_\delta$ and $\mathcal{C}_{\delta,\eta}$ given in Lemma~\ref{lem:HN} cannot be expected to be optimal, since the Hardy and Nash inequalities used in the proof have different optimizing functions. Here we shall present an alternative proof of~\eqref{Ineq:Hardy-Nash}, showing that the optimal value for $\mathcal{C}_\delta$ can be given in terms of the optimal constant of an appropriately chosen Caffarelli-Kohn-Nirenberg inequality of Nash type.

We start by rewriting~\eqref{CKN1} with optimal constant $\mathcal C=\mathcal C_{\text{CKN}}$ as
\[\label{CKN1.0}
\(\ird{|v|^2\,|x|^\beta}\)^{1+\frac2d}\le\mathcal C_{\text{CKN}}^{1+\frac2d}\ird{|\nabla v|^2\,|x|^\beta}\(\ird{|v|\,|x|^{\beta/2}}\)^\frac4d\,,
\]
which holds for $\beta>-d$. A straightforward computation shows that with the change of variables $v(x)=|x|^{-\beta/2}\,u(x)$, this is equivalent to~\eqref{Ineq:Hardy-Nash} with $\delta=-\,\beta^2/4-\,\beta\,(d-2)/2$. Thus, the choice $\beta=2-d + \sqrt{(d-2)^2-4\,\delta} > -d$ amounts to~\eqref{Ineq:Hardy-Nash} with optimal constant $\mathcal{C}_\delta=\mathcal{C}_{\text{CKN}}^{1+2/d}$.

%%%%%%%%%%%%%%%%%%%%%%%%%%%%%%%%%%%%%%%%%%%%%%%%%%%%%%%%%%%%%%%%%%%%%%
%%%%%%%%%%%%%%%%%%%%%%%%%%%%%%%%%%%%%%%%%%%%%%%%%%%%%%%%%%%%%%%%%%%%%%
\bigskip\noindent{\bf Acknowledgments:} This work has been partially supported by the Projects EFI ANR-17-CE40-0030 (E.B., J.D.) of the French National Research Agency. The work of C.S. has been supported by the Austrian Science Foundation (grants no. F65 and W1245), by the Fondation Sciences Math\'ematiques de Paris, and by Paris Sciences et Lettres. All authors are part of the Amadeus project \emph{Hypocoercivity} no.~39453PH.
\\\noindent{\scriptsize\copyright\,2019 by the authors. This paper may be reproduced, in its entirety, for non-commercial purposes.}

%%%%%%%%%%%%%%%%%%%%%%%%%%%%%%%%%%%%%%%%%%%%%%%%%%%%%%%%%%%%%%%%%%%%%%
%%%%%%%%%%%%%%%%%%%%%%%%%%%%%%%%%%%%%%%%%%%%%%%%%%%%%%%%%%%%%%%%%%%%%%
%\bibliographystyle{AIMS}\small\bibliography{References}
%\bibliographystyle{siam}\small\bibliography{References}

\bigskip\begin{center}\rule{2cm}{0.5pt}\end{center}\bigskip
\end{document}